\documentclass[a4paper,11pt]{amsart}
\usepackage[utf8]{inputenc}
\usepackage[english]{babel}
\usepackage{amsfonts,amsthm,amsmath,amssymb}
\usepackage[enableskew]{youngtab}
\usepackage{graphicx}
\usepackage{geometry}
\usepackage{todonotes}
\usepackage{tikz}
\usepackage{epigraph}
\usepackage{hyperref}
\usepackage{xy,xypic}
\newcommand{\CC}{{\mathbb{C}}}

\newcommand{\HH}{{\mathbb{H}}}
\newcommand{\QQ}{{\mathbb{Q}}}
\newcommand{\PP}{{\mathbb{P}}}
\newcommand{\RR}{{\mathbb{R}}}
\newcommand{\ZZ}{{\mathbb{Z}}}

\renewcommand{\Im}{\mathop{\mathrm{Im}}}

\newcommand{\SL}{\mathop{\mathrm{SL}}}

\newcommand{\PSL}{\mathop{\mathrm{PSL}}}
\newcommand{\PGL}{\mathop{\mathrm{PGL}}}

\renewcommand{\phi}{\varphi}

\theoremstyle{plain}
\newtheorem{theorem}{Theorem}[section]
\newtheorem{prop}[theorem]{Proposition}
\newtheorem{lemma}[theorem]{Lemma}
\newtheorem{corollary}[theorem]{Corollary}

\theoremstyle{remark}
\newtheorem{remark}[theorem]{Remark}
\newtheorem*{hint}{Hint}

\theoremstyle{definition}
\newtheorem{definition}[theorem]{Definition}
\newtheorem{example}[theorem]{Example}

\newtheorem{exercise}[theorem]{Exercise}

\numberwithin{equation}{section}
\numberwithin{figure}{section}

\begin{document}

\title{Friezes and Continued Fractions}
\author{Evgeny Smirnov}
\address{HSE University, Russia; Independent University of Moscow, Russia; Guangdong Technion -- Israel Institute of Technology, China}
\email{evgeny.smirnov@gmail.com}
\date{\today}
\dedicatory{To the memory of John Conway}

\maketitle

\begin{abstract} We explore basic properties of number friezes, due to Conway and Coxeter, and their relations to decompositions of rational numbers into continued fractions, Farey sequences, and the modular group acting on the hyperbolic plane. These are notes from a mini-course for undergraduate students given at the 19th Summer School ``Modern Mathematics'', Dubna, Russia, July 18--29, 2019.
\end{abstract}

\setcounter{tocdepth}{1}
\tableofcontents
\newpage

\epigraph{
\emph{\emph{Faust.}\\
Das Pentagramma macht dir Pein?\\
Ey sage mir, du Sohn der H\"olle,\\
Wenn das dich bannt, wie kamst du denn herein?\\
Wie ward ein solcher Geist betrogen?\\
\emph{Mephistopheles.}\\
Beschaut es recht! es ist nicht gut gezogen;\\
Der eine Winkel, der nach au\ss en zu,\\
Ist, wie du siehst, ein wenig offen.\\
\emph{Faust.}\\
Das hat der Zufall gut getroffen!\\
Und mein Gefangner w\"arst denn du?\\
Das ist von ohngef\"ahr gelungen!\\
~\\
\emph{(J.\,W.~von Goethe. Faust)}}}

\epigraph{
\emph{\emph{Faust:}\\
The pentagram prohibits thee?\\
Why, tell me now, thou Son of Hades,\\
If that prevents, how cam'st thou in to me?\\
Could such a spirit be so cheated?\\
\emph{Mephistopheles:}\\
Inspect the thing: the drawing's not completed.\\
The outer angle, you may see,\\
Is open left --- the lines don't fit it.\\
\emph{Faust:}\\
Well, --- Chance, this time, has fairly hit it!\\
And thus, thou'rt prisoner to me?\\
It seems the business has succeeded.\\
~\\
\emph{(Translated by Bayard Taylor)}
}}

\section*{Introduction}

Friezes are tables filled with positive numbers according to certain simple rules. They were first defined in the 1970s in works by H.\,S.\,M.\,Coxeter and John Conway. Despite the simplicity of their definition, friezes possess many surprising properties. They are connected with numerous other mathematical concepts including polygon triangulations, Catalan and Fibonacci numbers, continued fractions, Farey sequences... Moreover, it turns out that friezes (though not called by this name at the time) were already considered in works by Carl Friedrich Gauss and John Napier.

In the 2000s, the interest in friezes among mathematicians grew significantly due to the emergence of cluster algebra theory, introduced by Andrei Zelevinsky and Sergei Fomin. It turns out that friezes appear in connection with quiver representations, Grassmannians, elliptic functions, generalized associahedra, and many other topics from modern mathematics.

\bigskip

In these notes, we begin with classical questions about friezes considered by Coxeter and Conway. The first three chapters are devoted to this. In Chapter 1, we present basic definitions, make several observations about friezes (which will be proved later), and discuss the early history of the subject: how friezes appeared in works by Gauss and his predecessors.

In Chapter 2 we explore the phenomenon of integrality of friezes. We show that the elements of a frieze are computed as integer-coefficient polynomials, known as continuants, from the elements of the initial row. In Chapter 3, we continue studying friezes and derive relations between elements in their rows and diagonals. The main result of Chapter 4 is the classification of all integer friezes. As we will show, integer friezes of order $n$ are in bijection with triangulations of a convex $n$-gon; moreover, all elements of the frieze can be reconstructed from this triangulation using a simple algorithm.

Chapters 5 and 6 are devoted to expansions of rational numbers into two types of continued fractions: ``ordinary'' or positive, and negative (where the integer part is taken with excess rather than deficit). We formulate properties of continued fractions in terms of $2\times 2$ matrices and in Chapter 6 establish connections between the two types of expansions. We also show that the continued fraction expansion of a rational number corresponds to a triangulation of a polygon of a special form. It turns out that various characteristics of this continued fraction can be reconstructed from the triangulation: convergents, length of the continued fraction, etc. Thus continued fractions turn out to be connected with friezes: each rational number determines an integer frieze with the same triangulation.

In Chapters 7 and 8, we study the relation of these topics to  the action of the modular group $\PSL_2(\ZZ)$ on the hyperbolic plane. Namely, any triangulation of the described form can be obtained as a subgraph in the so-called Farey graph: an infinite graph whose vertices are rational points on the absolute of the hyperbolic plane, and the edges correspond to lines connecting points $p/q$ and $r/s$ such that $ps-qr=\pm 1$. This graph is preserved by the modular group $\PSL_2(\ZZ)$. We describe a construction that assigns to each rational number (or the corresponding polygon triangulation) an element of the modular group with an explicit expression in terms of  its standard generators. Finally, we examine how triangulations of polygons and their generalizations, 3D-dissections, are connected to relations on the generators of the modular group.

\bigskip

The dependency diagram of chapters looks as follows:

\[
\xymatrix{
1\ar[r] & 2\ar[r] &3\ar[r] &4 \ar[rd]\\
&&&&7 \ar[r] &8\\
&& 5 \ar[r] &6 \ar[ru] \\
}
\]

Readers more interested in continued fractions may skip the first four chapters and start with Chapters 5 and 6, that are almost independent of the previous part.

\bigskip

These notes are based on a mini-course of three lectures I gave at the 19th Summer School ``Modern Mathematics'' in Dubna in July 2019. I decided to speak about friezes after attending the talks by Valentin Ovsienko and Sophie Morier-Genoud at the Summer School ``Representation Theory of Lie Groups, Mathematical Physics and Combinatorics'' (Reims, France) in June of the same year. I also used materials from my mini-course on continued fractions and their connection with the modular group, given at the online school ``Combinatorics and Algorithms'' in February 2021. I am grateful to the organizers and participants of these schools.

Needless to say, these notes make no claim to completeness in covering the literature on friezes; for this I recommend the reader to turn to surveys and specialized articles on the subject. For example, one might start with a wonderful survey by  Sophie Morier-Genoud~\cite{Morier15}. Moreover, a collection of references on friezes and their connections with other areas of mathematics can be found on the homepage of Anna Felikson:

\medskip

\noindent\url{http://www.maths.dur.ac.uk/users/anna.felikson/Projects/frieze/frieze-res.html}

\medskip

I am grateful to the colleagues who read preliminary versions of these notes and provided feedback: Anna Felikson, Victor Kleptsyn, Grigory Merzon, Valentin Ovsienko, Sergei Tabachnikov, and others. I received particularly valuable feedback from Alexei Ustinov; from him I also learned about representing continuants using Morse codes. Finally, I want to thank my wife Svetlana Bochaver for her constant support and encouragement; without her, this text would likely never have been completed.
\bigskip

On April 11, 2020, as I was finishing the initial version of these notes, John Horton Conway, one of the most remarkable and extraordinary mathematicians of our time,  passed away at the age of 83. I dedicate this text to his memory.

\section{Friezes}

\subsection{Definition and first properties}\label{ssec:first}

Consider the following simple arithmetic game. We start with a row of $1$'s, going infinitely in both directions. Then we take a finite sequence of several nonnegative integers, not containing two subsequent $1$'s, extend it periodically in both directions, and write it under the first row with a half-step shift, so that the numbers would be arranged in a checkerboard pattern. These are the first two rows of our table.

Let us fill the subsequent rows of the table, starting  from the third one, in such a way that the numbers in these rows satisfy the \emph{unimodularity condition}: for any four numbers at the vertices of a unit diamond
\begin{equation}\label{eq:diamond}
\begin{array}{ccc}
&b&\\a &&d\\&c
\end{array}
\end{equation}
we have $ad-bc=1$. The game is terminated if at some step we get a row consisting only of $1$'s; in this case the table is said to be a \emph{frieze}.



Here is an example. Let us write in the second row of the table the following sequence: $(2,1,4,2,1,3,2)$, extended to both directions periodically. Then at each position at the third row we will have the product of the numbers above it minus one. This means that the third row will consist of repeating sequences $(1,3,7,1,2,5,3)$. It is clear that all the subsequent rows also have the same period as the second row, namely, 7. Computing this table further, we get the following:
\[
\begin{array}{ccccccccccccccccccccc}
\dots &1 && 1&& 1&& 1&& 1&& 1&& 1&& 1&& 1 &\dots\\
&\dots & 2 && 1 && 4 && 2 && 1 && 3 &&2 && 2 && 1& \dots\\
\dots &3 && 1&& 3&& 7&& 1&& 2&& 5&& 3&& 1 &\dots\\
&\dots & 1 && 2 && 5 && 3 && 1 && 3 &&7 && 1 && 3& \dots\\
\dots &2 && 1&& 3&& 2&& 2&& 1&& 4&& 2&& 2 &\dots\\
&\dots &1 && 1&& 1&& 1&& 1&& 1&& 1&& 1&& 1 &\dots\\
\end{array}
\]

We see that the sixth row contains only of $1$'s. According to the unimodality condition, the row below it would consist only of zeroes; we will not write it here. Sometimes it will also be convenient to suppose that there is an ``invisible'' row of zeroes above the first line; then the rows number zero, one and two will also satisfy the unimodality condition.


It is obvious that a frieze is determined by its second row. Further we will write something like ``frieze with the second row $(2,1,4,2,1,3,2)$'' (this is the frieze from the previous example), providing the period of the second row.

But if we put arbitrary nonnegative integers numbers into the second row, most probably we will not get a frieze. For example, take the second row consisting only of $2$'s. Then all the elements of the third row will be equal to $\frac{2\cdot 2-1}{2}=3$, the elements in the fourth row equal $\frac{3\cdot 3-1}{2}=4$, and so on; the $n$-th row will consist of $n$'s. So the entries in the rows of this table will be always increasing and thus will never become equal to 1. Sometimes we will abuse our convention and refer to such infinite tables satisfying the unimodulality condition as to \emph{infinite friezes}.

Looking at the previous example, we can notice some interesting properties of friezes. In particular, all elements of both finite and infinite friezes turn out to be integers. This is completely non-obvious: the unimodularity condition for the diamond~(\ref{eq:diamond}) implies that $c = (ad-1)/b$, meaning that to calculate the elements of the next row, one must use division; yet its result is  for some reason always an integer. In the next chapter, we will find out why this is the case.

Now let us return to ``finite'' friezes, that  is, those that end with a row of ones. From the construction, it is clear that if the second row of the frieze is periodic with period $n$, then all other rows will have the same property. However, the frieze from the given example additionally possesses glide symmetry: if we flip it upside down and shift it sideways by three and a half positions, it will coincide with the original. In Chapter~\ref{sec:relations}, we will show that this is a general property: a frieze with $n-1$ rows always possesses glide symmetry of order $n/2$ and, consequently, is periodic with period $n$.

Furthermore, one can consider friezes whose second row contains not necessarily integers, but arbitrary positive numbers. It turns out that the previous properties will also hold for them: the presence of glide symmetry and the connection between the number of rows and the period of the frieze. However, we will mainly be interested in integer friezes.

\textbf{A note on terminology.} In architecture, a frieze is a decorative ornament in the form of a horizontal band. In mathematics, this word has one more meaning: \emph{friezes}  are figures on a plane that repeat in only one direction, i.e., they are mapped onto themselves by a parallel translation. The theory of discrete transformation groups studies \emph{frieze groups} (also known as \emph{wallpaper groups}), which describe the possible types of symmetries of such patterns. It turns out there are only seven such groups; interestingly, Conway\footnote{John Horton Conway (1937--2020) was a prolific British mathematician active in many areas, including group theory (Conway groups, the moonshine conjecture), knot theory (Conway polynomial), combinatorial game theory (surreal numbers), and more. He was widely known as a popularizer of mathematics, particularly as the creator of the ``Game of Life''. } and Coxeter\footnote{Harold Scott MacDonald (Donald) Coxeter (1907--2003) was a British-born Canadian geometer, sometimes considered  one of the greatest geometers of the 20th century. He was the author of numerous works on regular polytopes, discrete groups, and more.}, the authors of the concept of a numerical frieze, also studied these groups. However, wallpaper groups are not directly related to numerical friezes, and we will not consider them in these notes. Those wishing to learn more about this can refer to, for example, the Wikipedia article: \url{https://en.wikipedia.org/wiki/Frieze_group}.

\subsection{Friezes with two and three rows}

Let us examine the structure of friezes with small number of rows. A two-line frieze is unique by definition: it consists only of $1$'s.
\[
\begin{array}{ccccccccccccccccccccc}
\dots &1 && 1&& 1&& 1&& 1&& 1&& 1&& 1&& 1 &\dots\\
&\dots &1 && 1&& 1&& 1&& 1&& 1&& 1&& 1&& 1 &\dots\\
\end{array}
\]

Next, a frieze with three rows is of the following form:
\[
\begin{array}{ccccccccccccccccccccc}
\dots &1 && 1&& 1&& 1&& 1&\dots\\
& \dots & a_1 && a_2&& a_3&& a_4&&\dots \\
\dots &1 && 1&& 1&& 1&& 1 &\dots\\
\end{array}
\]
The unimodularity condition is equivalent to saying that the product of any two elements in the middle row equals~2. Therefore the numbers in this row are alternating:
\[
\begin{array}{ccccccccccccccccccccc}
\dots &1 && 1&& 1&& 1&& 1&\dots\\
& \dots & a_1 && a_2&& a_1&& a_2&&\dots \\
\dots &1 && 1&& 1&& 1&& 1 &\dots\\
\end{array},
\]
with $a_1a_2=2$. In particular, there are two integer three-row friezes: one corresponding to $a_1=2$ and $a_2=1$, and the other to $a_1=1$ and $a_2=2$. These friezes differ only by a shift, but we will treat them as different ones. And if $a_1=a_2$, we get the following non-integer frieze:
\[
\begin{array}{ccccccccccccccccccccc}
\dots &1 && 1&& 1&& 1&& 1&\dots\\
& \dots & \sqrt{2} && \sqrt{2} && \sqrt{2} &&  \sqrt{2} &&\dots \\
\dots &1 && 1&& 1&& 1&& 1 &\dots\\
\end{array}
\]
\begin{remark}\label{rem:period}
Note that three-row friezes are periodic with period 2, not 4 (and the period of the last frieze is just 1). This is related to the glide symmetry: if a three-row frieze is flipped upside down and shifted by half a period, that is, by two, it should coincide with the initial one.
\end{remark}

\subsection{Four-row friezes} It is more convenient to construct such friezes from the diagonal rather than from the top row. Consider a frieze of the form:
\[
\begin{array}{ccccccccccccccccccccc}
\dots &1 && 1&& 1&& 1& \dots\\
&\dots & x&& z && v && t&\dots\\ 
&&\dots & y && u && w& \dots\\
&\dots & 1&& 1&& 1&& 1& \dots\\
\end{array}
\]
Let us express all its values consequently in terms of $x$ and $y$. We get:
\[
z=\frac{1+y}{x};\qquad u=\frac{1+z}{y}=\frac{1+x+y}{xy};\qquad v=\frac{1+u}{z}=\frac{1+x+y+xy}{xy}\cdot\frac{x}{1+y}=\frac{1+x}{y};
\]
\[
w=\frac{1+v}{u}=\frac{1+x+y}{y}\cdot\frac{xy}{1+x+y}=x;\qquad t=\frac{1+w}{v}=(1+x)\cdot\frac{y}{1+x}=y.
\]
We obtain $x$ and $y$ again, but on a different diagonal. We see that this frieze again turns out to be periodic with period 5 and possesses glide symmetry.

\begin{exercise}
List all integer friezes of order 5. How many are there?
\end{exercise}

\begin{small}
\begin{remark} If we consider the  elements of a frieze as rational functions of \emph{variables} $x$ and $y$, then the denominator of these rational functions always turns out to be a monomial. In other words, the frieze elements are not just rational functions, but \emph{Laurent polynomials}: polynomials where negative integer powers of variables are also allowed. Since we often divide by non-monomials, we can note that ``mysteriously'' the division always turns out to be exact, though now in the ring of Laurent polynomials. This observation, known as the \emph{Laurent phenomenon}, has far-reaching generalizations in cluster algebra theory.
\end{remark}
\end{small}

\begin{exercise}
Express the elements of a frieze with five rows through the values on the diagonal. Will the Laurent phenomenon hold in this case?
\end{exercise}

\begin{exercise}
Consider a frieze with four rows with $x=y=\tau$. Show that $\tau=\frac{1+\sqrt{5}}{2}$, and find the remaining elements of this frieze.
\end{exercise}

The last frieze can be obtained from a regular pentagon as follows.
\begin{figure}[h!]
\begin{tikzpicture}[scale=2.5,line cap=round,line join=round]
\foreach \x in {0,...,4} {
  \coordinate (A\x) at (\x*360/5-90-180/5:1);
  \draw[thick] (\x*360/5-90-180/5:1)--(\x*360/5-90+180/5:1);
}
\draw (A4)--(A0) (A0)--(A1) (A1)--(A2) (A2)--(A3) (A3)--(A4);
\draw (A4)--(A1) (A2)--(A4) (A1)--(A3) (A2)--(A0) (A3)--(A0);

\node[below] at (A0) {$B$};
\node[below] at (A1) {$C$};
\node[right] at (A2) {$D$};
\node[above] at (A3) {$E$};
\node[left]  at (A4) {$A$}; 
\end{tikzpicture}
\caption{Regular pentagon}\label{fig:pentagon}
\end{figure}
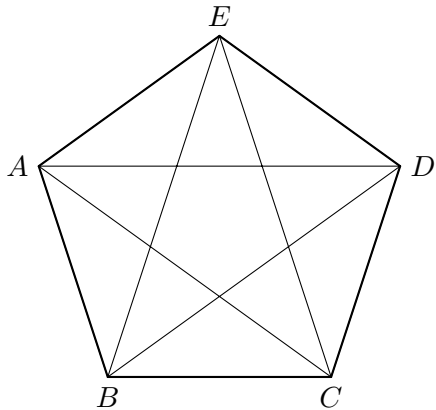
Consider a regular pentagon with all sides equal to 1 (Fig.~\ref{fig:pentagon}). Denote its diagonal by $\tau$. 
\emph{Ptolemy's theorem\footnote{Claudius Ptolemy (c. 100 -- c. 170 AD) was an ancient Greek mathematician, astronomer, astrologer, geographer, and music theorist. In his famous astronomical treatise \emph{the Almagest}, he proposed a geocentric model of the universe, in which the motion of celestial bodies is represented as a combination of uniform circular motions. He compiled a table of chords corresponding to angles in whole degrees and introduced the division of degrees into minutes and seconds.}} from plane geometry states that for an inscribed quadrilateral $ABCD$, the equality $AB\cdot CD+AD\cdot BC=AC\cdot BD$ holds. 
We can write Ptolemy's theorem for an isosceles trapezoid $ABCD$: in it $AB=BC=CD=1$, and $AD=AC=BD=\tau$. We get the equation
\[
1\cdot\tau + 1\cdot 1=\tau\cdot\tau,\qquad\text{i.e.}\qquad \tau^2-\tau-1=0.
\]
The positive root of this equation is exactly $\frac{1+\sqrt{5}}{2}$; this is the famous \emph{golden ratio}.

\begin{exercise}
Consider a regular $n$-gon with side length 1 and repeatedly apply Ptolemy's theorem to it. What frieze will we obtain? What are the elements in its $k$-th row?
\end{exercise}

\subsection{Pentagramma mirificum}

Friezes of order 5 first appeared in connection with spherical trigonometry in the works of Nathaniel Torporley\footnote{Nathaniel Tarporley, also Torporley (1564--1632). English mathematician, astronomer and astrologer. He worked for some time in France as a secretary of Fran\c{c}ois Vi\`ete.} and John Napier\footnote{John Napier (1550--1617). Scottish mathematician, physisist and astronomer, inventor of logarithms and the decimal point.}. They were further developed in the works of Carl Friedrich Gauss\footnote{Carl Friedrich Gauss (1777–1855) was a German mathematician, physicist, and astronomer. He is credited with fundamental results in algebra, geometry, number theory, analytical mechanics, and the theory of magnetism. He was called \emph{Princeps Mathematicorum}, the Prince of Mathematicians. The monument to Gauss in his native Braunschweig is adorned with a 17-pointed star, as Gauss was the first to describe the construction of a regular 17-sided polygon using a compass and straightedge.}. Their appearance is related to the following problem in spherical geometry.

Consider a five-pointed star on the unit sphere where all angles are right angles (see Fig.~\ref{fig:pentagram}). In particular, this means that point $A'$ is one of the poles for the great circle $CD$, while $B'$ is a pole for the great circle $DE$, etc.

\begin{figure}[h!]
\includegraphics[width=7cm]{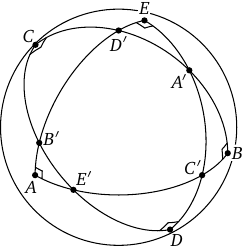}\quad
\includegraphics[width=7.5cm]{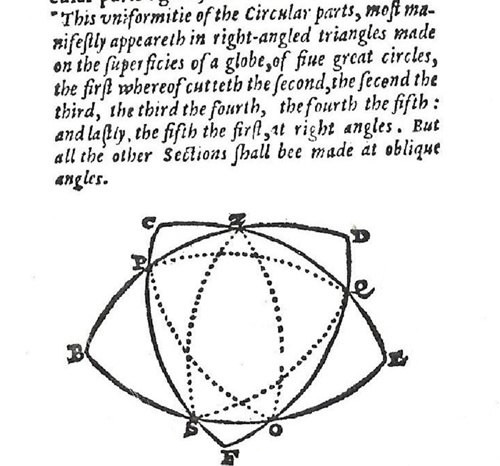}
\caption{Pentagramma mirificum. Right: illustration from Napier's work, 1614.}\label{fig:pentagram}
\end{figure}

Denote the squares of tangents of the side lengths of pentagon $A'B'C'D'E'$  by $\alpha,\dots,\epsilon$:
\[
\alpha=\tan^2 B'E',\quad \beta=\tan^2 E'C',\quad \gamma=\tan^2 A'C',\quad \delta=\tan^2A'D',\quad \epsilon=\tan^2D'B'.
\]
It turns out that these five quantities are related by the following equations, which allow obtaining the other three from any two of them:
\begin{equation}\label{eq:angles}
1+\alpha=\gamma\delta;\qquad 1+\beta=\delta\epsilon;\qquad
1+\gamma=\epsilon\alpha;\qquad 1+\delta=\alpha\beta;\qquad 1+\epsilon=\beta\gamma.
\end{equation}

\begin{exercise}[optional, for spherical geometry enthusiasts] Verify these formulas.
\end{exercise}

Thus, these quantities are elements of the following order 5 frieze:
\[
\begin{array}{ccccccccccccccccccccc}
\dots &1 && 1&& 1&& 1&& 1&& 1& \dots\\
&\dots & \alpha && \beta  && \gamma && \delta &&\epsilon &\dots\\ 
&&\dots & \delta && \epsilon && \alpha &&\beta &&\gamma & \dots\\
&\dots & 1&& 1&& 1&& 1&& 1&& 1& \dots\\
\end{array}
\]

\begin{exercise} Using relations (\ref{eq:angles}), prove the following equality (also due to Gauss):
\[
\alpha\beta\gamma\delta\epsilon = 3+\alpha+\beta+\gamma+\delta+\epsilon=\sqrt{(1+\alpha)(1+\beta)(1+\gamma)(1+\delta)(1+\epsilon)}.
\]
Gauss referred to this equality as to \emph{sch\"one Gleichung}, ``a beautiful equation''.
\end{exercise}

\begin{small}
In more modern terms, we can say that friezes describe the \emph{moduli space} $\mathcal{M}_{0,5}$ of five points on a projective line $\PP^1$. Specifically, consider five distinct points $z_1,\dots,z_5$ on $\PP^1$, up to the action of $\PGL(2)$. For any four of them we can write their cross-ratio. Let
\[
u_1=[z_1:z_2:z_3:z_4]=\frac{(z_3-z_1)(z_4-z_2)}{(z_3-z_2)(z_4-z_1)},\quad u_2=[z_2:z_3:z_4:z_5],\quad\dots
\]
and so on. Thus we obtain five cross-ratios $u_1,\dots,u_5$. It can be shown that they satisfy relations similar to (\ref{eq:angles}). Therefore, such quintuples (and, consequently, friezes with four rows) parameterize five points on $\PP^1$ up to a projective coordinate change. The details are left as an exercise for the interested reader.
\end{small}


\section{Continuants}

\subsection{Recurrence relation}\label{ssec:rec}

Now let us consider friezes with arbitrarily many rows. Our immediate goal is to obtain an expression for elements of the third and subsequent rows of a frieze in terms of elements of the second row.

Consider a frieze with second row $\dots,a_1,a_2,a_3,\dots$. Clearly, each element of the third row can be expressed through two elements of the second row that are directly above it (left and right). Furthermore, each element of the fourth row can be found by the unimodularity rule through \emph{three} elements of the second row above it, and so on. Let us denote elements of the third row by $a_{12},a_{23},\dots$, the fourth row by $a_{13}$, $a_{24}$, etc., as shown below:
\[
\begin{array}{ccccccccccccccccccccc}
\dots &1 && 1&& 1&& 1& &1&\dots\\
&\dots & a_1 && a_2  && a_3 &&a_4 & \dots\\
&& \dots & a_{12} && a_{23} && a_{34} &\dots\\
&&&\dots & a_{13}& &a_{24}& \dots \\
&&&&\dots & a_{14} &\dots\\
\end{array}
\]

It is easy to express elements of the third row:
\[
a_{12}=a_1a_2-1,\qquad a_{23}=a_2a_3-1.
\]
From them we can express elements of the fourth row:
\[
a_{13}=\frac{a_{12}a_{23}-1}{a_2}=a_1a_2a_3-a_1-a_3,
\]
the fifth row:
\[
a_{14}=\frac{a_{13}a_{24}-1}{a_{23}}=\dots=a_1a_2a_3a_4-a_1a_2-a_1a_4-a_3a_4+1,
\]
the sixth row:
\[
a_{15}=a_1a_2a_3a_4a_5-a_1a_2a_3-a_1a_2a_5-a_1a_4a_5-a_3a_4a_5+a_1+a_3+a_5,
\]
and so on (do these calculations yourself!).

Looking carefully at these expressions, we can notice a pattern they follow.

\subsection{Morse codes}\label{ssec:morse}

Consider $n$ points in a row, labeled by variables $a_1,\dots,a_n$. Now connect some pairs of adjacent points with the condition that each point can be connected to at most one neighbor. We obtain a ``Morse code'' of dots and dashes: a configuration like the following one.
\[
\stackrel{a_1}{\circ}-\stackrel{a_2}{\circ}\quad\stackrel{a_3}{\circ}\quad\stackrel{a_4}{\circ}\quad\stackrel{a_5}{\circ}-\stackrel{a_6}{\circ}\quad\stackrel{a_7}{\circ}
\]

To such a configuration we can associate a monomial: the product of variables corresponding to \emph{dots}, taken with a ``$+$'' or ``$-$'' sign depending on the parity of the number of \emph{dashes}. We will call this monomial the \emph{weight} of the Morse code $M$ and denote it by $w(M)$. Thus, the configuration in the previous figure corresponds to the expression $w(M)=(-1)^2a_3a_4a_7=a_3a_4a_7$.

\begin{example} All Morse codes and their corresponding monomials for $n$ from 1 to 5 are shown in Table~\ref{tab:morse}.
\end{example}

\begin{table}
\begin{tabular}{|l|c|c|}
\hline
$n=1$ & $\circ$ & $a_1$\\
\hline
$n=2$ & ${\circ}\quad{\circ}$ & $a_1a_2$\\
 & ${\circ}-{\circ}$ & $-1$\\
 \hline
$n=3$ & ${\circ}\quad{\circ}\quad{\circ}$ & $a_1a_2a_3$\\
 & ${\circ}\quad{\circ}-{\circ}$ & $-a_1$\\
  & ${\circ}-{\circ}\quad{\circ}$ & $-a_3$\\
\hline
$n=4$ &
$\circ\quad\circ\quad\circ\quad\circ$ & $a_1a_2a_3a_4$\\
&  $\circ\quad\circ\quad\circ-\circ$ & $-a_1a_2$\\
& $\circ\quad\circ-\circ\quad\circ$ & $-a_1a_4$\\\
& $\circ-\circ\quad\circ\quad\circ$ &$-a_3a_4$\\
& $\circ-\circ\quad\circ-\circ$ & $1$\\
\hline
$n=5$ &
$\circ\quad\circ\quad\circ\quad\circ\quad\circ$ & $a_1a_2a_3a_4a_5$\\
&$\circ\quad\circ\quad\circ\quad\circ- \circ$ & $-a_1a_2a_3$
\\
&$\circ\quad\circ\quad\circ-\circ\quad\circ$ & $-a_1a_2a_5$\\
&$\circ\quad\circ-\circ\quad\circ\quad\circ$ & $-a_1a_4a_5$\\
&$\circ-\circ\quad\circ\quad\circ\quad\circ$ & $-a_3a_4a_5$\\
&$\circ\quad\circ-\circ\quad\circ-\circ$ & $a_1$\\
&$\circ-\circ\quad\circ\quad\circ-\circ$ & $a_3$\\
&$\circ-\circ\quad\circ-\circ\quad\circ$ & $a_5$\\
\hline
\end{tabular}
\caption{Morse codes for $n\leq 5$}\label{tab:morse}	
\end{table}

\begin{definition}\label{def:euler} The \emph{continuant $V_n(a_1,\dots,a_n)$ of order $n$} is defined as the sum of monomials corresponding to all possible Morse codes on $n$ vertices.	
\end{definition}

\begin{remark}
We set the zero order continuant (that does not depend upon any arguments) to be $V_0=1$. Moreover, it is sometimes convenient to consider $V_n=0$ when $n$ is negative.
\end{remark}

Table~\ref{tab:morse} allows us to compute continuants of orders up to 5. We see that we get exactly the same expressions as for the frieze elements computed in the previous section. We will soon prove this claim, but first let us establish some more properties of continuants. We  start with a recurrence relation.

\begin{prop}\label{prop:contrecurrence}
Continuants satisfy the recurrence relation
\begin{equation}\label{eq:cont}
V_n(a_1,\dots,a_n)=a_nV_{n-1}(a_1,\dots,a_{n-1})-V_{n-2}(a_1,\dots,a_{n-2})
\end{equation}
with the initial condition $V_0=1$, $V_1=a_1$.
\end{prop}

\begin{proof}
Take any Morse code on $n$ vertices. Look at its last vertex. There are two possibilities. Either it corresponds to a dot, in which case this configuration $M$ corresponds to a Morse code $M'$ on vertices $1,\dots,{n-1}$; here $w(M)=a_nw(M')$, since in code $M$ the last vertex is counted while in configuration $M'$ it is not. Or, conversely, the last vertex is connected to the penultimate one by a dash; in this case we remove these two vertices and denote the resulting Morse code on vertices $1,\dots,n-2$ by $M''$. Then the weights of codes $M$ and $M''$ will differ by sign: $w(M)=-w(M'')$.

Thus, all codes starting with a dot will contribute $a_nV_{n-1}(a_1,\dots,a_{n-1})$ to the continuant $V_n(a_1,\dots,a_n)$, while those starting with a dash will contribute $-V_{n-2}(a_1,\dots,a_{n-2})$. This gives us the desired relation~(\ref{eq:cont}).
\end{proof}

We leave further properties of continuants as exercises for the reader.

\begin{exercise}\label{ex:continuant} Prove the following identities:
\begin{enumerate}
    \item $V_n(a_1,\dots,a_n)=V_n(a_n,\dots,a_1)$;
    \item $V_n(a_1,\dots,a_{n-1},0)=-V_{n-2}(a_1,\dots,a_{n-2})$;
    \item $V_n(a_1,\dots,a_n)=a_1V_{n-1}(a_2,\dots,a_{n})- V_{n-2}(a_3,\dots,a_{n})$.
\end{enumerate}
\end{exercise}

\subsection{Unimodularity relation for continuants}\label{ssec:unimod}

\begin{theorem}\label{thm:unimod}
For any $a_1,\dots,a_{n+1}$, the following relation holds:
\begin{equation}\label{eq:unimod}
    V_{n}(a_1,\dots,a_n)V_n(a_2,\dots,a_{n+1})=V_{n-1}(a_2,\dots,a_n)V_{n+1}(a_1,\dots,a_{n+1})+1.
\end{equation}
\end{theorem}

We provide two proofs of this theorem: an algebraic and a combinatorial one.

\begin{proof}[Algebraic proof]
We prove the theorem by induction on $n$. The base case for $n=1$ is obvious.

To prove the induction step, write the required relation for $n$ and rewrite all terms involving $a_{n+1}$ using relation~(\ref{eq:cont}):
\[
V_{n+1}(a_1,\dots,a_{n+1})=a_{n+1}V_n(a_1,\dots,a_n)-V_{n-1}(a_1,\dots,a_{n-1});
\]
\[
V_n(a_2,\dots,a_{n+1})=a_{n+1}V_{n-1}(a_2,\dots,a_n)-V_{n-2}(a_2,\dots,a_{n-1}).
\]
Then the equality becomes:
\begin{multline*}
V_{n}(a_1,\dots,a_n)(a_{n+1}V_{n-1}(a_2,\dots,a_n)-V_{n-2}(a_2,\dots,a_{n-1}))=\\=V_{n-1}(a_2,\dots,a_n)(a_{n+1}V_n(a_1,\dots,a_n)-V_{n-1}(a_1,\dots,a_{n-1}))+1.
\end{multline*}

Both sides of the equality contain the term $a_{n+1}V_{n-1}(a_2,\dots,a_n)V_{n}(a_1,\dots,a_n)$; canceling it out, we obtain:
\[
-V_{n}(a_1,\dots,a_n)V_{n-2}(a_2,\dots,a_{n-1})=-V_{n-1}(a_1,\dots,a_{n-1}))V_{n-1}(a_2,\dots,a_n)+1,
\]
which holds by the induction hypothesis. The theorem is proved.
\end{proof}

\begin{proof}[Combinatorial proof]
    Consider each continuant in the left-hand side of~(\ref{eq:unimod}) as a sum of monomials corresponding to Morse codes on the sets $\{1,\dots,n\}$ and $\{2,\dots,n+1\}$, respectively. We represent these Morse codes on a single diagram, using dashed lines for the first code and solid lines for the second, as shown below. Thus, dashed lines may appear in all positions except the last one, while solid lines may appear in all positions except the first one. Therefore, the left-hand side of~(\ref{eq:unimod}) equals the sum of monomials obtained from all possible \emph{pairs} of Morse codes under these conditions.

\begin{figure}[h!]
\[
\xymatrix{\circ\ar@{--}[r] &\circ\ar@{-}[r] &\circ\ar@{--}[r] &\circ &\circ &\circ\ar@<-2pt>@{--}[r]\ar@{^-}@<2pt>[r]&\circ&\circ\ar@{-}[r]&\circ}
\]
\caption{A pair of Morse codes}\label{fig:morse1}
\end{figure}
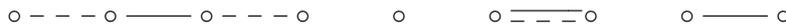

    Do the same for the product of continuants on the right-hand side of~(\ref{eq:unimod}): each product of monomials from $V_{n-1}(a_2,\dots,a_n)$ and $V_{n+1}(a_1,\dots,a_{n+1})$ can be interpreted as a pair of Morse codes on $\{1,\dots,n+1\}$, where the first code (dashed lines) cannot have lines in the first and last positions, while the second (solid lines) has no such restrictions.

    Consider any configuration $(M_1,M_2)$, where $M_1$ and $M_2$ are Morse codes on $\{1,\dots,n\}$ and $\{2,\dots,n+1\}$, respectively. It corresponds to the monomial $w(M_1)w(M_2)$ from the left-hand side of~(\ref{eq:unimod}). From it, construct a new configuration $(M_1',M_2')$, where $M_1'$ is a Morse code on $\{1,\dots,n+1\}$ and $M_2'$ is a code on $\{2,\dots,n\}$.
    
    If there is no (dashed) line between $1$ and $2$ in $M_1$, then $M_1$ can be viewed as a code on vertices $2,\dots,n$. Meanwhile, the solid-line code $M_2$ can be viewed as a code on vertices $1,\dots,n+1$ with no line in the first position. In this case, set $(M_1',M_2')=(M_1,M_2)$. 
    
Now suppose vertices 1 and 2 are connected by a dashed line in $M_1$. Since $M_2$ includes vertices starting from the 2nd, there is no solid line between vertices 1 and 2. Construct $(M_1',M_2')$ from $(M_1,M_2)$ as follows: consider the longest path starting from vertex 1 and consisting of alternating dashed and solid lines. Replace each dashed line with a solid one and each solid line with a dashed one. This yields a new pair $(M_1',M_2')$, where the first position has a solid line and no dashed line. This pair has the same weight as $(M_1,M_2)$. 

Figure~\ref{fig:morse2} shows the pair of codes obtained by this operation from those in Figure~\ref{fig:morse1}.
\begin{figure}[h!]
\[
\xymatrix{\circ\ar@{-}[r] &\circ\ar@{--}[r] &\circ\ar@{-}[r] &\circ &\circ &\circ\ar@<-2pt>@{--}[r]\ar@{^-}@<2pt>[r]&\circ&\circ\ar@{-}[r]&\circ}
\]
\caption{Result of applying the involution to the pair of Morse codes}\label{fig:morse2}
\end{figure}
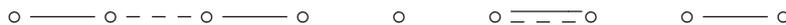

The pair $(M_1',M_2')$ will ``almost always'' correspond to a monomial from the right-hand side. Moreover, this procedure can ``almost always'' be reversed: from ``almost'' every $(M_1',M_2')$, we can uniquely reconstruct $(M_1,M_2)$, with equal weights. 

It remains to clarify what does ``almost always'' mean. We distinguish between the two cases: when $n$ is even and when it is odd.

If $n$ is even, the only configuration $(M_1,M_2)$ for which this procedure fails (i.e., no corresponding right-hand monomial exists) is:
\[
\xymatrix{\circ\ar@{--}[r] &\circ\ar@{-}[r] &\circ\ar@{--}[r] &\circ\ar@{-}[r] &\circ\ar@{--}[r] &\circ}
\]
Here, \emph{all} vertices are connected by alternating dashed and solid lines, with both first and last lines dashed (the number of vertices is even, so the number of lines between them is odd), and the weight of $(M_1,M_2)$ is $(-1)^{n-1}=-1$. Thus, the product from the left-hand side of~(\ref{eq:unimod}) is less by 1 than the product from the right-hand side.

If $n$ is odd, each $(M_1,M_2)$ has a unique image $(M_1',M_2')$, and exactly one right-hand configuration has no preimage:
\[
\xymatrix{\circ\ar@{--}[r] &\circ\ar@{-}[r] &\circ\ar@{--}[r] &\circ\ar@{-}[r] &\circ}
\]
In this case, the product in the right-hand side of~(\ref{eq:unimod}) is again greater by 1 than the product from the left-hand side.
\end{proof}

\subsection{Integer friezes}\label{ssec:integer}

The results of the previous section immediately yield the following theorem.

\begin{theorem} An element $a_{1n}$ in the $(n+1)$-th row of a frieze having numbers $a_1,\dots,a_n$ in the second row above it is equal to the continuant $V_n(a_1,\dots,a_n)$.    
\end{theorem}

This theorem explains the aforementioned ``integrality phenomenon'': all elements of a frieze with an integer second row $(a_1,a_2,\dots)$ turn out to be integers, even though their computation would seemingly require division.

\begin{corollary}
    A frieze whose second row elements are positive integers is entirely integer-valued.
\end{corollary}

Indeed, a continuant is a polynomial in $a_1,a_2,\dots$ with integer coefficients.

\subsection{Euler's identity}The theorem from the previous section admits a generalization known as \emph{Euler's identity\footnote{Leonhard Euler (1707–1783) was a Swiss, Prussian, and Russian mathematician. He authored more than 800 works across a vast range of fields, including mathematics, mechanics, physics, astronomy, and other sciences. He is considered one of the greatest mathematicians in history.} for continuants}.

\begin{theorem}\label{thm:euler}
    For any $m,\ell,n$, the following identity holds:
\begin{multline*}
V_{m+n}(a_1,\dots,a_{m+n})V_\ell(a_{m+1},\dots,a_{m+l})-V_{m+\ell}(a_1,\dots,a_{m+\ell})V_{n}(a_{m+1},\dots,a_{m+n})+\\
+V_{m-1}(a_1,\dots,a_{m-1})V_{n-\ell-1}(a_{m+\ell+2},\dots,a_{m+n})=0.
\end{multline*}
\end{theorem}

When $m=1$ and $n=\ell-1$, this precisely gives the unimodularity relation for continuants.

We will not use Euler's identity further and present it here without proof. It can be proved using double induction on two of its three parameters (this is left to the reader as an optional exercise). A short and elegant proof of this identity, due to A.\,Ustinov, can be found in his paper~\cite{Ustinov06}.
\bigskip

{\small
In the remaining part of this chapter we discuss the connection between continuants and determinants and explain alternative methods for obtaining the results mentioned above. Readers may skip this part without loss of continuity.

\subsection{Determinantal Expression for Continuants}\label{ssec:det} Continuants admit another useful expression: an $n$-th order continuant can be represented as the determinant of a tridiagonal matrix of the same order.

\begin{prop}\label{prop:contdet} Continuants admit the following determinant expression:
\begin{equation}\label{eq:continuant}
V_n(a_1,\dots,a_n)=\begin{vmatrix} a_1 &1 & 0 &\dots & 0 &0\\
1& a_2 & 1 &\dots &0 &0\\
0 & 1 & a_3 & \dots & 0& 0\\
\vdots &\vdots &\vdots &&\vdots &\vdots\\
0 & 0 & 0 &\dots & a_{n-1} & 1\\
0 & 0 & 0 &\dots &1& a_{n} \\
\end{vmatrix}.
\end{equation}   
\end{prop}

\begin{proof} We prove this equality by induction. The base cases for $n=0$ and $n=1$ are obvious. Expanding the determinant in the right-hand side of~(\ref{eq:continuant}) along the last row yields exactly the recurrence relation~(\ref{eq:cont}).
\end{proof}

This representation of continuants also leads to the unimodularity condition. For this, we need a matrix identity known as the \emph{Lewis Carroll identity}\footnote{Lewis Carroll is a penname of Charles Lutwidge Dodgson (1832--1898), English mathematician, writer and photographer, author of ``Alice's Adventures in Wonderland'' and ``Through the Looking-Glass". The identity bearing his name was used in his 1865 work, though it had appeared earlier in works by Desnanot and Jacobi.

Pierre Desnanot (1777--1857) was a French high school teacher and inspector. He authored textbooks on arithmetic, as well as works on the application of mathematics to mechanical problems. He headed the Clermont-Ferrand school district.}.

Let $A=(a_{ij})$ be an arbitrary $n\times n$ matrix. Denote its determinant by $M$; furthermore, let $M_{i_1,\dots,i_k}^{j_1,\dots,j_k}$ denote its $(n-k)$-th order minor obtained by deleting rows $i_1,\dots,i_k$ and columns $j_1,\dots, j_k$. The Lewis Carroll identity establishes a relation between $M$ and minors of orders $n-1$ and $n-2$:
\begin{theorem}[The Lewis Carroll identity]\label{thm:lcidentity}
The following equality holds:
\begin{equation}\label{eq:lce}
M\cdot M_{1,n}^{1,n}=M_1^1\cdot M_n^n-M_n^1\cdot M_1^n.
\end{equation}
\end{theorem}

We omit the proof here; the reader may devise it themself or look up in the literature. Note the combinatorial proof by D.\,Zeilberger\footnote{Doron Zeilberger (b. 1950) is an Israeli--American mathematician specializing in combinatorics and the theory of hypergeometric series. In 1995, he proved the famous alternating sign matrix conjecture by Mills, Robbins, and Rumsey. He has written a number of papers ``in collaboration'' with Shalosh B. Ekhad, his computer.} \cite{Zeilberger97}; it is conceptually similar to the proof of Theorem~\ref{thm:unimod} given above. 

\subsection{Computing continuants} \label{ssec:cont}

Let us compute the determinant $M=V_n(a_1,\dots,a_n)$ using the Lewis Carroll identity~(\ref{eq:lce}). A remarkable property of continuants is that the minors appearing in this  identity are either ones or lower-order continuants:
\[
M_1^1=V_{n-1}(a_2,\dots,a_n);\qquad M_n^n=V_{n-1}(a_1,\dots,a_{n-1});\qquad M_{1,n}^{1,n}=V_{n-2}(a_2,\dots,a_{n-1});
\]
while $M_1^n$ and $M_n^1$ are equal to 1 as the determinants of upper- and lower-triangular matrices with ones on the diagonal, respectively. Thus, the Lewis Carroll identity can be rewritten as follows:
\[
V_n(a_1,\dots,a_n)\cdot V_{n-2}(a_2,\dots,a_{n-1})=V_{n-1}(a_1,\dots,a_{n-1})\cdot V_{n-1}(a_2,\dots,a_{n})-1.
\]
In other words,
\begin{multline*}
1=V_n(a_1,\dots,a_n)\cdot V_{n-2}(a_2,\dots,a_{n-1})-V_{n-1}(a_1,\dots,a_{n-1})\cdot V_{n-1}(a_2,\dots,a_{n})=\\
=\begin{vmatrix}
V_{n-1}(a_1,\dots,a_{n-1}) & V_{n-2}(a_2,\dots,a_{n-1})\\
V_n(a_1,\dots,a_n) & V_{n-1}(a_2,\dots,a_{n})
\end{vmatrix}
.
\end{multline*}

This is exactly the unimodularity condition. Thus we obtain another proof of Theorem~\ref{thm:unimod}.

}


\section{Relations between frieze elements}\label{sec:relations}

\subsection{Relations between the second row and the Diagonal}

Like in the previous chapter, we start with a frieze of the form:
\[
\begin{array}{ccccccccccccccccccccc}
\dots &1 && 1&& 1&& 1& &1&\dots\\
&\dots & a_1 && a_2  && a_3 &&a_4 & \dots\\
&& \dots & a_{12} && a_{23} && a_{34} &\dots\\
&&&\dots & a_{13}& &a_{24}& \dots \\
&&&&\dots & a_{14} &\dots\\
&&&&&\dots&a_{15}&\dots
\end{array}
\]

We would like to establish relations between the elements of its second row $a_1,a_2,\dots$ and the diagonal elements $1,a_1,a_{12},a_{13},\dots$ (we will further refer to this diagonal as to the first one). We already know the relation expressing $a_{1k}$ through $a_1,\dots,a_k$: this is the continuant. However, continuants satisfy the recurrence relation given in Proposition~\ref{prop:contrecurrence}. This allows us to find the diagonal elements of the frieze one by one if its second row is known.

\begin{theorem}\label{thm:diagfrom2} The diagonal and top row elements of a frieze satisfy the relation:
\[
a_{1,k}=a_{k}a_{1,k-1}-a_{1,k-2}.
\]
\end{theorem}

\begin{remark}
Sometimes it is said that the frieze diagonal $a_1,a_{12},a_{13},\dots$ is a solution to the \emph{Sturm--Liouville difference equation}:
\[
V_k=a_kV_{k-1}-V_{k-2},
\]
where $a_1,a_2,\dots$ are coefficients and $V_1,V_2,\dots$ are unknowns.
\end{remark}

Furthermore, the previous theorem can be reformulated in the following way, expressing the second row elements through the diagonal elements.

\begin{theorem}\label{thm:2fromdiag}
For $k\geq 2$, the following equalities hold:
\[
a_k=\frac{a_{1,k-2}+a_{1k}}{a_{1,k-1}}.
\]
\end{theorem}

Note that here we never used the fact that our frieze is finite, i.e., that it ends with a row of $1$'s. So all this remains valid for infinite ``friezes'' (see Sec.~\ref{ssec:first}) with an arbitrary second row.

\subsection{Symmetries of finite friezes}

Now suppose our frieze is \emph{finite}: there exists an $n$ such that for all $i$, the equalities $a_{i,n+i-2}=1$ and $a_{i,n+i-1}=0$ hold. In other words, the $(n-1)$-th row consists entirely of ones, and the $n$-th row consists entirely of zeroes. Let us write Theorem~\ref{thm:diagfrom2} for the element $a_{1n}=0$:
\[
0=a_{1n}=a_na_{1,n-1}-a_{1,n-2}=a_n-a_{1,n-2}
\]
(recall that $a_{1,n-1}=1$ since the $(n-1)$-th row consists entirely of ones), from which we obtain:
\[
a_{1,n-2}=a_n.
\]

But the same relation holds for any diagonal, not just the first one. This gives us a set of equalities:
\[
a_{k,n-3+k}=a_{n+k-1}.
\]
Thus, the second row of the frieze coincides with the penultimate row, shifted right by $n/2$ positions. In this case, the frieze has the form:

\begin{equation}\label{eq:frieze}
\begin{array}{ccccccccccccccccccccc}
\dots &1 && 1&& 1&& 1&& 1 &&1&&\dots\\
&\dots & a_1 && a_2  && a_3 &&a_4&& \dots&&  a_{n-1}&&a_n&&\dots\\
&&\dots & a_{12} &&\dots &&&&&& \dots \\
&&&\dots & a_{13} &&\dots&&&& \dots \\
&&&&\dots & \dots &&\dots &&\dots\\
&&&&&\dots & a_{n} && a_{n+1}&&a_{n+2}&&\dots \\
&&&&\dots & 1 && 1&& 1 && \dots 
\end{array}
\end{equation}

From this we obtain the following result.

\begin{theorem}
A frieze with $n-1$ rows possesses a glide symmetry: its $k$-th row shifted by $n/2$ positions coincides with the $(n-k)$-th row.
\end{theorem}

\begin{proof}
From the unimodularity condition, it follows that the frieze can be constructed not only from the second row downward but also from the penultimate row upward, with the same result.
\end{proof}

\begin{corollary}\label{cor:periodic}
A frieze with $n-1$ rows is periodic with period $n$.
\end{corollary}

\begin{proof}
Shifting the frieze by $n$ is the same as applying two glide symmetries with shift $n/2$.
\end{proof}

\begin{definition} A frieze consisting of $n-1$ rows is said to have \emph{order} $n$.
\end{definition}

According to Corollary~\ref{cor:periodic}, any frieze of order $n$ is periodic with the period length equal to $n$; however, this length is not necessarily minimal (see~Remark~\ref{rem:period} on p.~\pageref{rem:period}).

\begin{exercise}
What identities on continuants are equivalent to the symmetry of the frieze?
\end{exercise}

\subsection{Infinite extension of finite friezes} By definition, a frieze of order $n$ terminates with a row of ones and a row of zeros. That is, for the elements of its (say, first) diagonal, the following equalities hold:
\[
a_{1,n-1}=1, \qquad a_{1,n}=0.
\]
But what happens if we extend this sequence further? Using the equality from Theorem~\ref{thm:diagfrom2}, we can define $a_{1,i}$ for $i>n$. We see that
\[
a_{1,n+1}=a_{n}a_{1,n}-a_{1,n-1}=a_n\cdot 0-1=-1.
\]
Extending this diagonal further, we get
\[
a_{1,n+2}=a_{n+1}a_{1,n+1}-a_{1,n}=-a_{n+1}=-a_1
\]
(here we used the periodicity of the second row with period $n$). We find that the $(n+1)$-th and $(n+2)$-th diagonal elements differ from the first and second only by sign! It is easy to see that this pattern continues: $a_{1,n+k}=-a_{1,k}$ for $k\leq n$. Thus, each frieze diagonal can be extended downward, and the next $n$ elements will differ from the first $n$ by multiplication by $-1$. Of course, this holds for any other diagonal as well. Similarly, we can say that for arbitrary $k$ we have
\[
a_{i,rn+k}=(-1)^r a_{ik}.
\]

Thus, we can extend our frieze below the row of zeros. Clearly, this extension will satisfy the unimodularity rule. The possibility to extend friezes downward will be useful in the next section.

\subsection{Transition to adjacent diagonals} In this section, we will examine how elements of three adjacent frieze diagonals are related. It turns out that knowing two of them allows us to recover the third  one through a simple matrix transformation.

As before, we consider the frieze:
\begin{equation*}
\begin{array}{ccccccccccccccccccccc}
\dots &1 && 1&& 1&& 1&&1&&\dots\\
&\dots & a_1 && a_2  && a_3 &&a_4&&\dots\\
&&\dots & a_{12} &&a_{23}&& a_{34}&&\dots  \\
&&&\dots & a_{13} && a_{24}&& a_{35}&&\dots\\
&&&&\dots &a_{14}&&\ddots &&\ddots\\
&&&&&& \ddots &&\ddots &&\ddots\\
&&&&&&\dots & a_{1,n-2} && a_{2,n-1}&&a_{3,n}&&\dots \\
&&&&&&&\dots & 1 && 1&& 1 && \dots 
\end{array}
\end{equation*}
To simplify the notation, we formally set $a_{kk}:=a_k$.

\begin{theorem}\label{thm:matrixshift} For any $k$, the following equality holds:
\[
\begin{pmatrix}
    a_1 &-1 \\1 &0
\end{pmatrix}
\begin{pmatrix}
    a_{2k} \\a_{3k}
\end{pmatrix}=
\begin{pmatrix}
    a_{1k} \\a_{2k}
\end{pmatrix}.
\]

\end{theorem}

Informally, applying such a matrix shifts the frieze diagonal running from southwest to northeast one position downward.

\begin{proof} We prove this equality by induction on $k$. For the base case, take $k=2$:
\[
\begin{pmatrix}
    a_1 &-1 \\1 &0
\end{pmatrix}
\begin{pmatrix}
    a_2\\ 1
\end{pmatrix}=
\begin{pmatrix}
    a_1a_2-1\\a_2
\end{pmatrix}=\begin{pmatrix}
    a_{12}\\a_2
\end{pmatrix}.
\]

Note that we could also take $k=1$ as the base case, assuming the frieze is extended upward with zeros; this would give the correct equality $\begin{pmatrix}
    a_1 &-1 \\1 &0
\end{pmatrix}
\begin{pmatrix}
    1\\ 0
\end{pmatrix}=
\begin{pmatrix}
    a_1\\1
\end{pmatrix}$.

To prove the inductive step, write the recurrence relations for the diagonal elements:
\begin{eqnarray*}
    a_{1,k+1}&=&a_{k+1}a_{1,k}-a_{1,k-1},\\
    a_{2,k+1}&=&a_{k+1}a_{2,k}-a_{2,k-1},\\
    a_{3,k+1}&=&a_{k+1}a_{3,k}-a_{3,k-1}.\\
\end{eqnarray*}
Note that all these equalities contain the same coefficient $a_{k+1}$.

Furthermore, the induction hypothesis states that
\begin{eqnarray*}
    a_{1,k}&=&a_1 a_{2,k}-a_{3,k},\\
    a_{1,k-1}&=&a_1 a_{2,k-1}-a_{3,k-1}.
\end{eqnarray*}
From this it follows that
\[
a_{1,k+1}=a_1 a_{2,k+1}-a_{3,k+1},
\]
as required.
\end{proof}

\begin{exercise}
Derive these equalities from Proposition~\ref{prop:contdet} by expanding the continuant determinant along the \emph{first} row.
\end{exercise}

Here we only used recurrence relations for diagonal elements, so these matrix equalities will also hold for elements of the frieze extended downward according to the rule from the previous section. From this we obtain a corollary that will be useful later when studying the connection between friezes and continued fractions. Let $I=\begin{pmatrix}
    1&0\\0&1
\end{pmatrix}$ be the $2\times 2$ identity matrix.

\begin{corollary}\label{cor:minusid}
    Let $a_1,\dots,a_n$ be the second row of a frieze of order $n$. Then the following matrix equality holds:
    \[
    \begin{pmatrix}
        a_1 & -1\\ 1&0
    \end{pmatrix}
        \begin{pmatrix}
        a_2 & -1\\ 1&0
    \end{pmatrix}\dots    \begin{pmatrix}
        a_n & -1\\ 1&0
    \end{pmatrix}=-I.
    \]
\end{corollary}

\begin{proof}
This follows from Theorem~\ref{thm:matrixshift} and the periodicity of the frieze. Indeed, shifting the frieze diagonal by $n$ positions yields the same diagonal but with a minus sign. Hence the composition of matrices on the left side of the equality acts as multiplication by minus identity.
\end{proof}


\section{Integer friezes and triangulations}

\subsection{The quiddity of an integer frieze}

In this chapter, we will describe all friezes consisting only of integer elements.

We will call the second row $(a_1,\dots,a_n)$ of an order $n$ frieze its \emph{quiddity}. Obviously, it cannot contain two consecutive 1's: this would contradict unimodularity. However, it turns out that at least one 1 must necessarily appear in this row.

\begin{prop} The second row of an integer frieze must contain at least one entry equal to~1.
\end{prop}

\begin{proof}
Suppose this is not the case, and there exists a frieze with all elements in the second row differ from 1: $a_k \geq 2$. Denote the elements of its first diagonal by $v_k = a_{1k}$. Then for them we have the inequality
\[
v_k = a_k v_{k-1} - v_{k-2} \geq 2v_{k-1} - v_{k-2},
\]
which implies that
\[
v_k - v_{k-1} \geq v_{k-1} - v_{k-2}.
\]
But this holds for all $k$. Therefore,
\[
v_k - v_{k-1} \geq v_{k-1} - v_{k-2} \geq \dots \geq v_2 - v_1 = a_2 - 1 \geq 1,
\]
meaning that the sequence $v_k$ is strictly increasing. This contradicts the assumption that at some point, due to the finite number of rows in the frieze, it must take the values 1 and 0.

This statement can also be proved geometrically. Suppose we found a frieze with all $a_k \geq 2$. Consider the cone $x \geq y \geq 0$ generated by vectors $\begin{pmatrix}1\\0\end{pmatrix}$ and $\begin{pmatrix}1\\1\end{pmatrix}$. A matrix of the form 
$\begin{pmatrix}a_k & -1\\1 & 0\end{pmatrix}$ transforms these vectors into $\begin{pmatrix}a_k\\1\end{pmatrix}$ and $\begin{pmatrix}a_k-1\\1\end{pmatrix}$ respectively. For $a_k \geq 2$, both these vectors lie in the interior of the cone $x \geq y \geq 0$. Therefore, the composition of such matrices $\begin{pmatrix}a_1 & -1\\1 & 0\end{pmatrix}\dots\begin{pmatrix}a_n & -1\\1 & 0\end{pmatrix}$ also maps this cone into its proper subset. But this contradicts Corollary~\ref{cor:minusid}, which states that this composition equals $-I$.
\end{proof}

The following lemma allows us to construct a new frieze of order $n+1$ from a given integer  frieze of order $n$.

\begin{lemma}[The van Gogh lemma]\label{lem:vangogh}
Let $(a_1,\dots,a_n)$ be the quiddity of an integer order $n$ frieze, and let $1 \leq k \leq n$. Then:

\begin{enumerate}
    \item The tuple $(b_1,\dots,b_{n+1}) = (a_1,\dots,a_{k-1}+1,1,a_k+1,a_{k+1},\dots,a_n)$ is the quiddity of an integer order $n+1$ frieze;
    \item If $1,v_1,v_2,\dots,v_{n-2}$ is the first diagonal of the original frieze starting with element \mbox{$v_1 = a_1$}, then the corresponding diagonal of the new frieze has the form
\[
1,v_1,v_2,\dots,v_{k-1},v_{k-1}+v_k,v_k,\dots,v_{n-2}.
\]
\end{enumerate}
\end{lemma}

\begin{proof}
We begin with part (2). Let us write down the first diagonal of the new frieze: denote it by $1,w_1,\dots,w_{n-1}$.

Clearly, for $i \leq k-2$, the diagonals of these two friezes coincide: $w_i = v_i$. At the $(k-1)$-th position of the new frieze we will have
\[
w_{k-1} = (a_{k-1}+1)v_{k-2} - v_{k-3} = v_{k-1} + v_{k-2}.
\]
Next,
\[
w_k = w_{k-1} - w_{k-2} = v_{k-1} + v_{k-2} - v_{k-2} = v_{k-1}.
\]
The next diagonal element will be
\[
w_{k+1} = (a_k+1)w_k - w_{k-1} = (a_k+1)v_{k-1} - v_{k-1} - v_{k-2} = a_k v_{k-1} - v_{k-2} = v_k.
\]
For all $j > k+1$, the equality $w_j = v_{j-1}$ holds. In particular, the $(n-2)$-th and $(n-1)$-th terms of this sequence are equal to 1 and 0 respectively, meaning the frieze terminates, and the diagonal contains $n$ positive elements. The same argument applies to any other diagonal. Thus, such an insertion yields an integer frieze of order one greater. Part (1) is also proved.
\end{proof}

\begin{remark}\label{rem:vangogh}
This construction is obviously reversible: if a sequence $(\dots,b_{k-1},1,b_{k+1},\dots)$ is the quiddity of an integer frieze, then $(\dots,b_{k-1}-1,b_{k+1}-1,\dots)$ will be the quiddity of a frieze of order one less (note that due to the absence of consecutive 1's, both numbers $b_{k-1}-1$ and $b_{k+1}-1$ are positive).
\end{remark}

\subsection{Friezes and triangulations}

Consider a convex $n$-gon with numbered vertices. Its \emph{triangulation} is its partition into triangles using diagonals that do not intersect except at vertices.

\begin{remark} The number of triangulations of an $n$-gon is called the \emph{$(n-2)$-th Catalan number} (this is one of their numerous equivalent definitions). For them, the formula $C(n) = \frac{1}{n-1}\binom{2n-4}{n-2}$ holds (see, for instance,~\cite{Stanley15}).
\end{remark}

Consider an arbitrary triangulation of an $n$-gon. We can associate to it a tuple of numbers $(c_1,\dots,c_n)$, where $c_i$ is the number of triangles incident to the $i$-th vertex. We will call such a tuple the \emph{quiddity of the triangulation}. The triangulation is uniquely determined by its quiddity (why?).

\begin{theorem}[J.\,Conway, H.\,S.\,M.\,Coxeter]\label{thm:cc}
The tuple $(c_1,\dots,c_n)$ constructed from a triangulation of an $n$-gon is the quiddity of an integer frieze of order $n$; this mapping establishes a bijection between triangulations of an $n$-gon and integer friezes of order $n$.
\end{theorem}

\begin{proof}
We will prove this theorem by induction on $n$. The base case $n=3$ is obvious: there is only one triangulation of a triangle and only one order $3$ frieze with quiddity $(1,1,1)$. (Those not convinced by this argument may start with $n=4$).

Consider a triangulation of an $n$-gon. Choose a vertex belonging to only one triangle. Such a vertex exists by the pigeonhole principle: if $n$ exterior sides belong to $n-2$ triangles of the triangulation, then there must be at least one triangle (in fact, at least two) to which two exterior sides belong. We will call such a triangle an ``ear".

Let the vertex belonging to only one triangle have number $k$, and let vertices $k-1$ and $k+1$ belong to $b_{k-1}$ and $b_{k+1}$ triangles respectively. Thus, the quiddity of this triangulation has the form $(\dots,b_{k-1},1,b_{k+1},\dots)$. Remove the triangle adjacent to the $k$-th vertex from this triangulation to obtain a triangulation of an $(n-1)$-gon with quiddity $(\dots,b_{k-1}-1,b_{k+1}-1,\dots)$. 
By the induction hypothesis, this triangulation determines a frieze with the same quiddity. But by Lemma~\ref{lem:vangogh}, then the tuple $(\dots,b_{k-1},1,b_{k+1},\dots)$ also determines an integer frieze.

The bijectivity of this correspondence follows from Remark~\ref{rem:vangogh}: from the quiddity of each  frieze of order $n$, we can remove a 1, reducing the problem to a frieze of order one less.
\end{proof}

\subsection{Reconstructing friezes from triangulations}\label{ssec:friezetriang}

Our next goal is to describe the combinatorial meaning of frieze elements obtained from a given triangulation. The Conway--Coxeter theorem states that each element of the frieze's second row is the number of triangles incident to the corresponding vertex of the triangulation. But how can we reconstruct the remaining frieze elements?

Consider an integer frieze, written with the notation of~(\ref{eq:frieze}) on p.~\pageref{eq:frieze}: its second row is $(a_1,a_2,\dots,a_n)$, and the elements of the diagonal starting at $a_1$ are denoted by $v_0=1,v_1,v_2,\dots,v_{n-2}$, i.e., $v_1=a_1$. It is also convenient to set $v_{-1}=0$.

Consider a triangulation of an $n$-gon whose vertices are numbered from $0$ to $n-1$, with quiddity $(a_1,\dots,a_n)$. We can reconstruct $(v_0,v_1,v_2,\dots)$ using the following algorithm:

\begin{itemize}
    \item write the numbers $0$ and $1$ in vertices $0$ and $1$ respectively;
    \item then place numerical labels in all other vertices of the polygon according to the following  rule: for each triangle where two vertices already have numbers $a$ and $b$, write $a+b$ in its third vertex. Repeat this procedure until all vertices are filled.
\end{itemize}

\begin{exercise}
Show that all vertices connected to vertex $0$ by a diagonal or side will have $1$ written in them.
\end{exercise}

\begin{prop}\label{prop:algtriang} The label obtained by this algorithm in the $i$-th vertex is equal to the diagonal element $v_{i-1}$ of the frieze.
\end{prop}

\begin{proof}
To this triangulation of the $n$-gon we can associate its dual graph with $n-2$ vertices: its vertices correspond to triangles in the triangulation, with two vertices connected by an edge if the corresponding triangles share a side. Such a graph is a tree, with all vertex degrees at most three. The leaves (degree 1 vertices) of this tree correspond to the ``ears'' of the triangulation.
\begin{figure}[h!]
\begin{tikzpicture}[scale=2.5,line cap=round,line join=round]
\foreach \x in {0,...,6} {
  \coordinate (A\x) at (\x*360/7-90-180/7:1);
  \draw[thick] (\x*360/7-90-180/7:1)--(\x*360/7-90+180/7:1);
}
\draw (A5)--(A0) (A0)--(A2) (A5)--(A2) (A2)--(A4);
\coordinate (B1) at (-.78,-.15);
\coordinate (B2) at (-.1,-.2);
\coordinate (B3) at (.3,-.7);
\coordinate (B4) at (0,.5);
\coordinate (B5) at (.6,.5);
\draw[dashed] (B1)--(B2) (B2)--(B3) (B2)--(B4) (B4)--(B5);
\foreach \x in {1,...,5} {
   \draw[fill=white] (B\x) circle (.3mm);
}
\end{tikzpicture}
\caption{A triangulation and its corresponding tree}\label{fig:triang}
\end{figure}
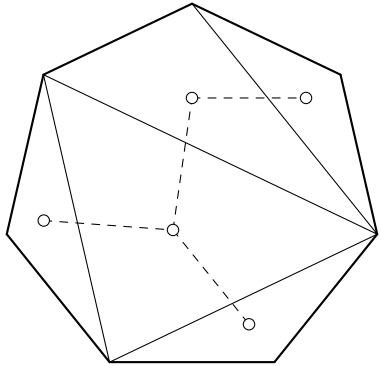

Next, writing numbers in triangle vertices according to our algorithm can be viewed as visiting the vertices of the corresponding tree: in the first step of the algorithm, we take the triangle containing vertices $0$ and $1$, and write a number in its third vertex (for the first step, this number will be $1$). Having done this, we mark the tree vertex corresponding to this triangle as visited.

At each subsequent step, we consider some vertex adjacent to already visited ones; it corresponds to a triangle where two vertices already have some numbers. Writing their sum in the triangle's third vertex, we mark the tree vertex corresponding to this triangle as visited.

We prove our proposition by induction on $n$. The base case $n=3$ is obvious: we obtain a triangle with $0$, $1$, and $1$ written in its vertices. It corresponds to the unique order $3$ frieze.

Inductive step: suppose our proposition is proved for all possible triangulations of an $n$-gon, and we want to prove it for an $(n+1)$-gon. Consider the last step of the algorithm; in its course, we must visit some pendant vertex of the triangulation tree, i.e., add an ``ear'' to the triangulation of the $n$-gon. Let this $n$-gon have quiddity $(a_1,\dots,a_n)$, and the ``ear'' at the last step is attached to side $(k-1,k)$. By the induction hypothesis, the resulting set of labels $(v_0,\dots,v_{n-1})$ in the vertices  lie on the diagonal of the frieze with quiddity $(a_1,\dots,a_n)$. Next, at the last step, a vertex with label $v_{k-1}+v_k$ is be added between vertices $k-1$ and $k$. The triangulation of the $(n+1)$-gon  then has quiddity $(\dots,a_{k-1}+1,1,a_k+1,\dots)$, which agrees with the result of Lemma~\ref{lem:vangogh} about frieze elements. The proposition is proved.
\end{proof}

\begin{exercise}
Consider a triangulation of an $n$-gon where all diagonals form a zigzag, as in Fig.~\ref{fig:octagon} (in other words, all elements of this triangulation's quiddity are at most 3). Prove that all elements of the corresponding frieze will be Fibonacci numbers.
\end{exercise}

\begin{figure}[h!]
\includegraphics{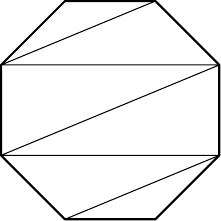}
\caption{A zigzag triangulation}\label{fig:octagon}
\end{figure}

\begin{exercise} Consider a triangulation containing a vertex where at least 4 triangles meet. Prove that the corresponding frieze contains the number 4.
\end{exercise}

\subsection{Admissible paths}
Another method for computing frieze elements, presented in~\cite{BrolineCroweIsaacs74}, is based on counting so-called admissible paths.

\begin{definition}
An \emph{admissible path} from vertex $i-1$ to vertex $j+1$ is an ordered sequence of distinct triangles $\tau_i,\dots,\tau_j$, where triangle $\tau_\ell$ is incident to vertex $\ell$.
\end{definition}

\begin{prop} The frieze element $a_{ij}$ equals the number of admissible paths from $i-1$ to $j+1$.
\end{prop}

We will not prove this proposition; readers may try to reconstruct the proof themselves or read it in~\cite{BrolineCroweIsaacs74}.

\begin{remark} Note that by definition, a path from vertex $(i-1)$ to vertex $(i+1)$ consists of one triangle passing through vertex $i$. Then the number of admissible paths between these vertices is simply $a_i$. Furthermore, the path from vertex $(i-1)$ to vertex $i$ is unique: it consists of \emph{zero} triangles, so all elements of the frieze's first row equal $1$.
\end{remark}

\begin{example} Consider the following triangulation of a heptagon.
\begin{figure*}[h!]
\includegraphics{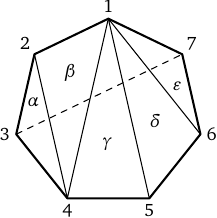}
\end{figure*}

Let $\alpha,\beta,\gamma,\delta,\varepsilon$ be the triangles in it. List all admissible paths $(\tau_1,\tau_2)$ between vertices $7$ and $3$. These paths are as follows:
\[
(\varepsilon,\beta), \quad(\varepsilon,\alpha),\quad(\delta,\beta),\quad
(\delta,\alpha),\quad (\gamma,\beta),\quad (\gamma,\alpha),\quad (\beta,\alpha).
\]
This gives $a_{12}=7$. (Note that in this example $a_1=4$, $a_2=2$).
\end{example}

\begin{exercise}[non-trivial] \textbf{a)} Show that an admissible path as a sequence is uniquely determined by its set of triangles.

\textbf{b)} Verify that the complement to a path from vertex $i$ to vertex $j$ determines a path from vertex $j$ to vertex $i$. For instance, in the above example if $(\tau_1,\tau_2)$ is an admissible path from $7$ to $3$, then $\{\alpha,\beta,\gamma,\delta,\epsilon\}\setminus 
\{\tau_1,\tau_2\}$ determines an admissible path from $3$ to $7$. Derive from this statement another proof of the glide symmetry of friezes.
\end{exercise}

\begin{remark} Most results from this section can be found in two papers by J.\,Conway and H.\,S.\,M.\,Coxeter~\cite{ConwayCoxeter73a}, \cite{ConwayCoxeter73b}. These papers, published in two consecutive issues of the educational journal \emph{The Mathematical Gazette}, are structured as problem sets: the first article presents a series of statements about friezes as a sequence of about thirty problems, while the second provides their solutions or sufficiently detailed hints.
\end{remark}


\section{Continued fractions}

In the next chapters, we will see how continuants, polygon triangulations, and friezes arise in connection with the classical problem of expanding a number into a continued fraction.

\subsection{Convergents}

Let $p/q > 1$ be an irreducible fraction. It can be expanded into a continued fraction as follows: take the integer part of $p/q$ by dividing $p$ by $q$ with remainder. Let $a_1$ be the quotient and $r$ the remainder; then
\[
\frac{p}{q} = a_1 + \frac{r}{q}.
\]
The fraction $q/r$, reciprocal to $r/q$, will now be greater than 1, so we can again extract its integer part, invert the fractional part, and so on. At each step, the denominators of the fractions decrease, so the process will eventually terminate (at some point the next fraction will be an integer). Thus we obtain the representation
\[
\frac{p}{q} = a_1 + \cfrac{1}{a_2 + \cfrac{1}{a_3 + \cfrac{1}{\dots + \cfrac{1}{a_n}}}}.
\]
For brevity, we will denote the right-hand side by $[a_1, a_2, a_3 \dots, a_n]$.

This representation is almost unique; the only ambiguity is that the last number can be represented either as $a_n$ or as $(a_n - 1) + \dfrac{1}{1}$. We will get rid of this ambiguity by requiring that the expansion of a rational number into a continued fraction must always have an \emph{even} number of components.

\begin{example}\label{ex:running} The expansion of $7/5$ into a continued fraction is:
\[
\frac{7}{5} = 1 + \cfrac{1}{2 + \cfrac{1}{1 + \cfrac{1}{1}}} = [1, 2, 1, 1].
\]
\end{example}

Suppose we are given the continued fraction expansion of a number: $p/q = [a_1, \dots, a_n]$. We can consider not the entire expansion, but only the part formed by its first $i$ terms: $p_i/q_i = [a_1, \dots, a_i]$. The result is called the \emph{$i$-th convergent} to $p/q$. For example, the convergents to $7/5$ from the previous example are $1$, $3/2$, and $4/3$.

By the way, note that the procedure of expanding a number into a continued fraction is exactly the Euclidean algorithm applied to the numerator and denominator of the original fraction.

\subsection{Matrices of continued fractions}

Consider a rational number written as an irreducible fraction $p/q > 1$. When we expand it into a continued fraction, at each step we extract its integer part (denoted by $a_1$) and invert the fractional part, thereby obtaining a new irreducible fraction $p'/q'$. Let us see how these numbers are related:
\[
\frac{p}{q} = a_1 + \frac{1}{p'/q'};\qquad \frac{p}{q} = \frac{a_1 p' + q'}{p_1}.
\]
From this we can write the equalities
\begin{eqnarray*}
p &=& a_1 p' + q';\\
q &=& p'.
\end{eqnarray*}
Note that here we used the irreducibility of the fractions $p/q$ and $p'/q'$.

Thus, the numbers $p$ and $q$ are obtained from $p'$ and $q'$ by a linear transformation given by the following matrix:
\[
\begin{pmatrix}p\\q
\end{pmatrix}=
\begin{pmatrix}a_1 & 1\\1 & 0
\end{pmatrix}
\begin{pmatrix}p'\\q'
\end{pmatrix}.
\]
We can repeat the same procedure by extracting the integer part $a_2$ from the fraction $p'/q'$; we obtain a new fraction $p''/q''$, which gives us the equality
\[
\begin{pmatrix}p\\q
\end{pmatrix}=
\begin{pmatrix}a_1 & 1\\1 & 0
\end{pmatrix}
\begin{pmatrix}a_2 & 1\\1 & 0
\end{pmatrix}
\begin{pmatrix}p''\\q''
\end{pmatrix}.
\]
This process terminates when the next fraction $p^{(n-1)}/q^{(n-1)}$ becomes an integer $a_n$, which is equivalent to $p^{(n-1)} = a_n$ and $q^{(n-1)} = 1$. We obtain the equality
\[
\begin{pmatrix}p\\q
\end{pmatrix}=
\begin{pmatrix}a_1 & 1\\1 & 0
\end{pmatrix}
\dots \begin{pmatrix}a_{n-1} & 1\\1 & 0
\end{pmatrix}
\begin{pmatrix}a_n\\1
\end{pmatrix}.
\]

We can also ``extract the fractional part'' from the integer $a_n$, obtaining the ``fraction'' $p^{(n)}/q^{(n)} = 1/0$. Of course, the rational number $1/0$ does not exist, but at the matrix level everything works fine:
\[
\begin{pmatrix}p\\q
\end{pmatrix}=
\begin{pmatrix}a_1 & 1\\1 & 0
\end{pmatrix}
\dots \begin{pmatrix}a_{n-1} & 1\\1 & 0
\end{pmatrix}
\begin{pmatrix}a_n & 1\\1 & 0
\end{pmatrix}
\begin{pmatrix}1\\0
\end{pmatrix}.
\]

What happens if we apply the same product of matrices not to the vector $\begin{pmatrix}1\\0\end{pmatrix}$, but to the other basis vector $\begin{pmatrix}0\\1\end{pmatrix}$? We obtain the equality
\[
\begin{pmatrix}r\\s
\end{pmatrix}=
\begin{pmatrix}a_1 & 1\\1 & 0
\end{pmatrix}
\dots \begin{pmatrix}a_{n-1} & 1\\1 & 0
\end{pmatrix}
\begin{pmatrix}a_n & 1\\1 & 0
\end{pmatrix}
\begin{pmatrix}0\\1
\end{pmatrix}=
\begin{pmatrix}a_1 & 1\\1 & 0
\end{pmatrix}
\dots \begin{pmatrix}a_{n-1} & 1\\1 & 0
\end{pmatrix}
\begin{pmatrix}1\\0
\end{pmatrix}.
\]

We have represented the vector with components $r$ and $s$ as a product of matrices of the form $\begin{pmatrix}a_i & 1\\1 & 0\end{pmatrix}$, but for the tuple $(a_1, \dots, a_{n-1})$ rather than $(a_1, \dots, a_n)$. It follows that $r/s$ equals the \emph{last convergent} $p_{n-1}/q_{n-1}$ to $p/q$:
\[
p/q = [a_1, \dots, a_n];\qquad r/s = [a_1, \dots, a_{n-1}].
\] 

\begin{definition} The \emph{continued fraction matrix} for the number $p/q = [a_1, \dots, a_n]$, $p/q > 1$, is the matrix
\[
M^+(a_1, \dots, a_n) = \begin{pmatrix}a_1 & 1\\1 & 0
\end{pmatrix}
\dots 
\begin{pmatrix}a_n & 1\\1 & 0
\end{pmatrix}.
\]
\end{definition}

We obtain the following proposition.

\begin{prop}\label{prop:matrixpos} Let $M^+(a_1, \dots, a_{2m})$ be the continued fraction matrix for $p/q$, and let $r/s = p_{2m-1}/q_{2m-1}$ be the last convergent to $p/q$. Then
\[
M^+(a_1, \dots, a_{2m}) = \begin{pmatrix}p & r\\q & s
\end{pmatrix}.
\]
\end{prop}

The same can be said for any pair of adjacent convergents, not necessarily the last.

\begin{corollary} Let $p_{i-1}/q_{i-1}$ and $p_i/q_i$ be the $(i-1)$-th and $i$-th convergents to $p/q = [a_1, \dots, a_n]$, respectively. Then
\begin{equation}\label{eq:contdecomp}
\begin{pmatrix}p_i & p_{i-1}\\q_i & q_{i-1}
\end{pmatrix}=
\begin{pmatrix}a_1 & 1\\1 & 0
\end{pmatrix}
\dots 
\begin{pmatrix}a_i & 1\\1 & 0
\end{pmatrix}.
\end{equation}
\end{corollary}

From this statement, the following properties of convergents immediately follow.

\begin{prop}\label{prop:contfrac} The numerators and denominators of the convergents $p_i/q_i = [a_1, \dots, a_i]$ to $p/q = [a_1, \dots, a_n]$ satisfy the following equalities:

\begin{enumerate}
\item $p_i q_{i-1} - p_{i-1} q_i = (-1)^i$;
\item $p_i$ and $q_i$ are coprime;
\item For $p_i$ and $q_i$, the recurrence relations $p_k = a_i p_{i-1} + p_{i-2}$ and $q_i = a_i q_{i-1} + q_{i-2}$ hold, with initial conditions $p_0 = 1$, $p_1 = a_1$, $q_0 = 0$, $q_1 = 1$;
\item The sequences of numerators and denominators of the convergents are increasing: \mbox{$p_{i-1} < p_i$} and $q_{i-1} < q_i$ for any $i$.
\end{enumerate}
\end{prop}

\begin{proof}
(1) Note that the determinant of a matrix of the form $\begin{pmatrix}a & 1\\1 & 0\end{pmatrix}$ equals $-1$. Therefore,
\[
\det \begin{pmatrix}p_i & p_{i-1}\\q_i & q_{i-1}\end{pmatrix} = (-1)^i.
\]
In particular, if we require the number of components in this expansion to be even, we obtain the equality $\det M^+(a_1, \dots, a_{2m}) = 1$.

(2) This follows directly from (1).

(3) Consider the last factor in the product on the right-hand side of equality (\ref{eq:contdecomp}):
\[
\begin{pmatrix}p_i & p_{i-1}\\q_i & q_{i-1}
\end{pmatrix}=
\begin{pmatrix}p_{i-1} & p_{i-2}\\q_{i-1} & q_{i-2}
\end{pmatrix}
\begin{pmatrix}a_i & 1\\1 & 0
\end{pmatrix}.
\]
For $i = 1$, we obtain the equality
\[
\begin{pmatrix}p_1 & p_0\\q_1 & q_0
\end{pmatrix}=
\begin{pmatrix}a_1 & 1\\1 & 0
\end{pmatrix},
\]
which gives us the initial conditions. For arbitrary $k$, the equality of the first columns of these matrices exactly corresponds to the recurrence relations.

Statement (4) follows from (3), since all $a_i$ are strictly positive.
\end{proof}

\begin{remark}
Of course, all the listed properties of convergents can be proved without resorting to matrix language. The reader can do this themself or read about it, for example, in the classic book by Alexander Khinchin ``Continued Fractions''~\cite{Khinchin97}. 
\end{remark}

\subsection{Continuants again}

We see that relation (3) from Proposition~\ref{prop:contfrac} is very similar to the recurrence relation for continuants (Proposition~\ref{prop:contrecurrence}): they differ only in sign. Therefore, we can define \emph{positive continuants} using the relation $K_0(\cdot) = 1$, $K_1(a_1) = a_1$, $K_i(a_1, \dots, a_i) = a_i K_{i-1}(a_1, \dots, a_{i-1}) + K_{i-2}(a_1, \dots, a_{i-2})$.

\begin{exercise} Formulate and prove analogs of Euler's rule (Definition~\ref{def:euler}) and the properties from Exercise~\ref{ex:continuant} for positive continuants and interpret them similarly to Proposition~\ref{prop:contdet} in terms of expanding determinants by row or column.
\end{exercise}

Then Proposition~\ref{prop:contfrac} (3) takes the following form:
\begin{corollary}
The $i$-th convergent to the fraction $p/q = [a_1, \dots, a_{2m}]$ equals the ratio of continuants
\[
\frac{p_i}{q_i} = \frac{K_i(a_1, \dots, a_i)}{K_{i-1}(a_1, \dots, a_{i-1})},
\]
and the fraction in the right-hand side is irreducible.
\end{corollary}

\subsection{Negative continued fractions}

When defining a continued fraction, we took the ``floor'' integer part of a number, that is, the largest integer not exceeding this fraction. What happens if we the integer part with an excess: the ``ceiling'' function,  and then apply the same procedure to the difference between it and the original number? We obtain an expansion of the number into a \emph{negative continued fraction}; sometimes it is also called a \emph{Hirzebruch continued fraction\footnote{Friedrich Hirzebruch (1927--2012): German mathematician, one of the founders of the Max Planck Institute for Mathematics in Bonn.}}. To avoid confusion, we will call the ``usual'' continued fraction a \emph{positive continued fraction}.

In this case, all components of the negative continued fraction for a number greater than 1 are natural numbers not less than 2; such an expansion is  unique. We denote it by double square brackets.
\[
\frac{p}{q} = c_1 - \cfrac{1}{c_2 - \cfrac{1}{c_3 - \cfrac{1}{\dots}}} = [[c_1, c_2, c_3 \dots]].
\]

\begin{example} The number $7/5$ expands into a negative continued fraction as
\[
\frac{7}{5} = 2 - \cfrac{1}{2 - \cfrac{1}{3}} = [[2, 2, 3]].
\]
\end{example}

For negative continued fractions, convergents are defined similarly. We will denote them by $\tilde{p}_i/\tilde{q}_i$. They are obtained by taking the first $i$ terms of the expansion: $\tilde{p}_i/\tilde{q}_i = [[c_1, \dots, c_i]]$. For example, the convergents to $7/5$ are $2/1$ and $3/2$.

Similarly, we can define negative continued fraction matrices. Let $p/q$ be the original fraction, and suppose that the fraction $\tilde{p}/\tilde{q}$ is obtained from it by subtracting the ceiling integer part, changing the sign, and taking the reciprocal. Since
\[
\frac{p}{q} = c_1 - \frac{1}{\tilde{p}/\tilde{q}} = \frac{c_1 \tilde{p} - \tilde{q}}{\tilde{p}},
\]
the transition between the previous and the next continued fractions is given by the matrix relation
\[
\begin{pmatrix}p\\q
\end{pmatrix}=
\begin{pmatrix}c_1 & -1\\1 & 0
\end{pmatrix}
\begin{pmatrix}\tilde{p}\\ \tilde{q}
\end{pmatrix}.
\]

Now we can give the definition of the negative continued fraction matrix and prove the following proposition, analogous to Proposition~\ref{prop:matrixpos}.

\begin{definition} The \emph{negative continued fraction matrix} for the number $p/q = [[c_1, \dots, c_k]]$, $p/q > 1$, is the matrix
\[
M(c_1, \dots, c_k) = \begin{pmatrix}c_1 & -1\\1 & 0
\end{pmatrix}
\dots 
\begin{pmatrix}c_k & -1\\1 & 0
\end{pmatrix}.
\]
\end{definition}

\begin{remark}
One of the advantages of working with matrices is that the matrices $M(c_1, \dots, c_k)$ turn out to be more ``universal'' than continued fractions: they are defined for any sequence of coefficients, whereas for continued fractions this is not the case, since some denominators may become zero. As an example, consider the following ``continued fraction":
\[
[[3, 1, 2, 1]] = 3 - \cfrac{1}{1 - \cfrac{1}{2 - \cfrac{1}{1}}}.
\]
It does not correspond to any rational number, while the matrix $M(3, 1, 2, 1)$ is well-defined (compute it!).
\end{remark}

\begin{prop}\label{prop:matrixneg} Let $M(c_1, \dots, c_k)$ be the negative continued fraction matrix for $p/q$, and let $\tilde{r}/\tilde{s} = \tilde{p}_{k-1}/\tilde{q}_{k-1}$ be the last negative convergent to $p/q$. Then
\[
M(c_1, \dots, c_k) = \begin{pmatrix}p & -\tilde{r}\\q & -\tilde{s}
\end{pmatrix}.
\]
\end{prop}

From this, the expression for the matrices of convergents immediately follows.

\begin{corollary}
Let $\tilde{p}_{i-1}/\tilde{q}_{i-1}$ and $\tilde{p}_i/\tilde{q}_i$ be the $(i-1)$-th and $i$-th negative convergents to $p/q = [[c_1, \dots, c_k]]$, respectively. Then
\begin{equation}\label{eq:negcontdecomp}
\begin{pmatrix}\tilde{p}_i & \tilde{p}_{i-1}\\\tilde{q}_i & \tilde{q}_{i-1}
\end{pmatrix}=
\begin{pmatrix}c_1 & -1\\1 & 0
\end{pmatrix}
\dots 
\begin{pmatrix}c_i & -1\\1 & 0
\end{pmatrix}.
\end{equation}
\end{corollary}

Thus, for negative convergents, there is an analog of Proposition~\ref{prop:contfrac}.

\begin{prop}\label{prop:negcontfrac}
The numerators and denominators of the negative convergents $\tilde{p}_i/\tilde{q}_i = [[c_1, \dots, c_i]]$ to $p/q = [[c_1, \dots, c_k]]$ satisfy the following equalities:

\begin{enumerate}
\item $\tilde{p}_i \tilde{q}_{i-1} - \tilde{p}_{i-1} \tilde{q}_i = 1$;
\item $\tilde{p}_i$ and $\tilde{q}_i$ are coprime;
\item For $\tilde{p}_i$ and $\tilde{q}_i$, the recurrence relations $\tilde{p}_i = c_i \tilde{p}_{i-1} - \tilde{p}_{i-2}$ and $\tilde{q}_i = c_i \tilde{q}_{i-1} - \tilde{q}_{i-2}$ hold, with initial conditions $\tilde{p}_0 = 1$, $\tilde{p}_1 = c_1$, $\tilde{q}_0 = 0$, $\tilde{q}_1 = 1$;
\item The sequences of numerators and denominators of the negative convergents are increasing: $\tilde{p}_{i-1} < \tilde{p}_i$ and $\tilde{q}_{i-1} < \tilde{q}_i$ for any $i$.
\end{enumerate}
\end{prop}

\begin{proof}
These statements are proved similarly to the positive case. Moreover, the situation here is even simpler: matrices of the form $\begin{pmatrix}c & -1\\1 & 0\end{pmatrix}$ always have determinant 1.

The only slight difference is in the proof of part (4). This statement will again follow from the recurrence relation (3). We prove it by induction. The base case is obvious. Induction hypothesis: suppose we know that $\tilde{p}_{i-1} > \tilde{p}_{i-2}$. Write the expression for $\tilde{p}_i$ and use the inequality $c_i \geq 2$. Then we have
\[
\tilde{p}_i = c_i \tilde{p}_{i-1} - \tilde{p}_{i-2} \geq 2 \tilde{p}_{i-1} - \tilde{p}_{i-2} = \tilde{p}_{i-1} + (\tilde{p}_{i-1} - \tilde{p}_{i-2}) \geq \tilde{p}_{i-1},
\]
as required. Similarly for $\tilde{q}_i$.
\end{proof}

Here we see the connection between continued fraction expansions and friezes: the numerators and denominators of convergents transform according to the same recurrence rules as the elements of friezes. For example, they are obtained as ratios of continuants.

\begin{prop} Let $p/q = [[c_1, \dots, c_k]]$ be the expansion of a rational number into a negative continued fraction. Then the $i$-th convergent $\tilde{p}_i/\tilde{q}_i = [[c_1, \dots, c_i]]$ to $p/q$ equals the ratio of continuants
\[
\frac{\tilde{p}_i}{\tilde{q}_i} = \frac{V_i(c_1, \dots, c_i)}{V_{i-1}(c_1, \dots, c_{i-1})}, 
\]
and the fraction in the right-hand side is irreducible.
\end{prop}

This is proved in exactly the same way as for positive continued fractions.

\subsection{Relation between positive and negative continued fraction expansion matrices}

We have described two ways to expand a rational number $p/q$ into a continued fraction: the positive $p/q = [a_1, \dots, a_{2m}]$ and the negative $p/q = [[c_1, \dots, c_k]]$. For each of these expansions, we constructed a matrix: $M^+ = M^+(a_1, \dots, a_{2m})$ and $M = M(c_1, \dots, c_k)$, respectively.

\begin{example}
For $p/q = 7/5$, we obtain:
\[
M^+(1, 2, 1, 1) = \begin{pmatrix}7 & 4\\5 & 3\end{pmatrix},
\qquad
M(2, 2, 3) = \begin{pmatrix}7 & -3\\5 & -2\end{pmatrix}.
\]
\end{example}

A natural question arises: how are these matrices related? The connection turns out to be very simple: their first columns are identical and equal to the difference between the second columns of $M^+$ and $M$. In the matrix language, this is formulated as follows.

\begin{prop}\label{pr:mplusmr} Let $M^+$ and $M$ be the positive and negative continued fraction matrices for $p/q$, respectively. Then
\[
M^+ = M \cdot \begin{pmatrix}1 & 1\\0 & 1\end{pmatrix}.
\]
\end{prop}

\begin{proof}
By Proposition~\ref{prop:contfrac}, the matrix $M^+$ equals $\begin{pmatrix}p & r\\q & s\end{pmatrix}$, where $r/s$ is the last positive convergent to $p/q$. Similarly, Proposition~\ref{prop:negcontfrac} states that $M = \begin{pmatrix}p & -\tilde{r}\\q & -\tilde{s}\end{pmatrix}$, where $\tilde{r}/\tilde{s}$ is the last negative convergent. Thus, the first columns of $M^+$ and $M$ are equal, and the required equality is equivalent to the statement that $r + \tilde{r} = p$ and $s + \tilde{s} = q$.

Let us prove these equalities. We use the unimodularity property for convergents. We know that
\[
\tilde{r} q - \tilde{s} p = 1.
\]
On the other hand,
\[
r q - s p = -1.
\]
The minus sign arises because the index of the last convergent is odd: by our definition, the number of components in the expansion of a number into a positive continued fraction is even. Thus,
\[
(r + \tilde{r}) q - (s + \tilde{s}) p = 0.
\]
But from the inequalities on the numerators and denominators of convergents, we know that $r < p$ and $\tilde{r} < p$, as well as $s < q$ and $\tilde{s} < q$. Therefore, $0 < r + \tilde{r} < 2p$ and $0 < s + \tilde{s} < 2q$. Hence, $r + \tilde{r} = p$ and $s + \tilde{s} = q$, as required.
\end{proof}


\section{Continued fraction matrices and the group $\SL_2(\ZZ)$}

\subsection{Relation between positive and negative continued fractions}

In the previous section, we established how the matrices of positive and negative continued fractions are related. This chapter answers a different question: suppose a fraction $p/q$ has both a ``positive'' expansion $[a_1,\dots,a_{2m}]$ and a ``negative'' expansion $[[c_1,\dots,c_k]]$. How are these sequences of numbers related? In other words, how can we reconstruct one sequence knowing the other? While the simplest approach would be to compute the fraction from one expansion and then expand it using the other method, is there a more direct connection? The following theorem provides the answer.

\begin{theorem}\label{thm:hirz} Let a rational number $p/q>1$ have positive and negative continued fraction expansions $[a_1,\dots,a_{2m}]$ and $[[c_1,\dots,c_k]]$ respectively. Then 
\[
(c_1,\dots,c_k)=(a_1+1,\underbrace{2,\dots,2}_{a_2-1},a_3+2,\underbrace{2,\dots,2}_{a_4-1},\dots,a_{2m-1}+2,\underbrace{2,\dots,2}_{a_{2m}-1}).
\]
\end{theorem}

We will prove this theorem in two ways. The first method, covered in this chapter, uses matrix computations. The second, more geometric approach is discussed in the next chapter. It is related to the action of the modular group on the hyperbolic plane.

\subsection{The group $\SL_2(\ZZ)$ and its generators}\label{ssec:sl2z}

Before proceeding with the first proof, we need to introduce some notation. By $\SL_2(\ZZ)$ we will denote the group of $2\times 2$ integer matrices with determinant 1. We have seen that continued fraction matrices $M(c_1,\dots,c_k)$ and $M^+(a_1,\dots,a_{2m})$ belong to this group.

Let $R$, $L$, and $S$ denote the following elements of $\SL_2(\ZZ)$:
\[
R=\begin{pmatrix}1 & 1\\ 0 &1
\end{pmatrix},\qquad
L=\begin{pmatrix}1 & 0\\ 1 &1
\end{pmatrix},\qquad
S=\begin{pmatrix}0 & -1\\ 1 &0
\end{pmatrix}.
\]

It can be shown (try this yourself or read in Chapter VII of J.-P. Serre's ``A Course in Arithmetic'' \cite{Serre70}) that these three matrices generate $\SL_2(\ZZ)$. In fact, any two of them suffice: each of these three matrices can be expressed in terms of the other two.

\begin{exercise} Show that $L=S^{-1}R^{-1}S$. Find expressions for $R$ in terms of $S$ and $L$, and for $S$ in terms of $R$ and $L$.
\end{exercise}

The element $S$ has order 4 (i.e., $S^4=I$), while $R$ and $L$ have infinite order. Indeed, their powers can be computed explicitly: $R^a=\begin{pmatrix}1 &a\\0&1\end{pmatrix}$ and $L^a=\begin{pmatrix}1 &0\\a&1\end{pmatrix}$ for any $a\in\ZZ$.

Continued fraction matrices can be expressed through these generators.

\begin{prop}
The matrices $M(c_1,\dots,c_k)$ and $M^+(a_1,\dots,a_{2m})$ admit the following decompositions:
\begin{equation}\label{eq:mplusgen}
M^+(a_1,\dots,a_{2m})=R^{a_1}L^{a_2}R^{a_3}L^{a_4}\dots R^{a_{2m-1}}L^{a_{2m}}
\end{equation}
and
\begin{equation}\label{eq:mgen}
M(c_1,\dots,c_k)=R^{c_1}S R^{c_2}S\dots R^{c_k}S.
\end{equation}
\end{prop}

\begin{proof} It is easy to see that
\[
\begin{pmatrix}a_i &1\\1 &0\end{pmatrix}=R^{a_i}\begin{pmatrix}0 &1\\1 &0\end{pmatrix}\quad\text{and}\quad
\begin{pmatrix}a_{i+1} &1\\1 &0\end{pmatrix}=\begin{pmatrix}0 &1\\1 &0\end{pmatrix}L^{a_{i+1}}.
\]
Therefore,
\[
\begin{pmatrix}a_i &1\\1 &0\end{pmatrix}
\begin{pmatrix}a_{i+1} &1\\1 &0\end{pmatrix}=R^{a_i}L^{a_{i+1}},
\]
which yields formula (\ref{eq:mplusgen}). Formula (\ref{eq:mgen}) is even simpler, it follows from the equality
\[
\begin{pmatrix}c_i &-1\\1 &0\end{pmatrix}=R^{c_i}S.
\]
\end{proof}

This leads to the following proposition.

\begin{prop}\label{pr:hirzlite}
The following equality holds:
\begin{equation}\label{eq:hirzlite}
M^+(a_1,\dots,a_{2m})=-M(a_1+1,\underbrace{2,\dots,2}_{a_2-1},a_3+2,\underbrace{2,\dots,2}_{a_4-1},\dots,a_{2m-1}+2,\underbrace{2,\dots,2}_{a_{2m}},1,1).
\end{equation}
\end{prop}

Note that if we had already proved Theorem~\ref{thm:hirz}, equality (\ref{eq:hirzlite}) could be derived from Proposition~\ref{pr:mplusmr}, which states that $M^+=MR$. However, we will proceed differently: first prove (\ref{eq:hirzlite}), then derive Theorem~\ref{thm:hirz} from it.

We will need the following lemma.

\begin{lemma}\label{lem:rala}
The equalities $R^a=-M(a+1,1,1)$ and $L^a=-M(1,\underbrace{2,\dots,2}_{a},1,1)$ hold.
\end{lemma}

\begin{proof}
The first equality is easy to verify directly.

For the second equality, first note that $M(1,1,1)=-I$, i.e., $\begin{pmatrix}1&-1\\1&0\end{pmatrix}^{-1}=-\begin{pmatrix}1&-1\\1&0\end{pmatrix}^2$. Next, it is easy to verify directly that $L=\begin{pmatrix}1&-1\\1&0\end{pmatrix}\begin{pmatrix}2&-1\\1&0\end{pmatrix}\begin{pmatrix}1&-1\\1&0\end{pmatrix}^{-1}$. Thus, $L^a$ is obtained from $M(2)^a=M(2,\dots,2)$ by conjugation by the same matrix, as required.
\end{proof}

\begin{proof}[Proof of Proposition~\ref{pr:hirzlite}]
Since $M(1,1,1)=-I$, Lemma~\ref{lem:rala} implies that
\[
R^{a_i}L^{a_{i+1}}=-M(a_i+1,\underbrace{2,\dots,2}_{a_{i+1}},1,1).
\]
Thus, the proposition follows from relation (\ref{eq:mplusgen}) and the simple equality $M(2,1,1,a+1)=-M(a+2)$ (we leave its verification as an exercise). The proposition is proved.
\end{proof}

Finally, note that the last three coefficients in (\ref{eq:hirzlite}) are $(2,1,1)$, so they can be eliminated using the relation $M(2,1,1)=-R$. Thus we obtain:
\begin{equation}
M^+(a_1,\dots,a_{2m})=M(a_1+1,\underbrace{2,\dots,2}_{a_2-1},a_3+2,\underbrace{2,\dots,2}_{a_4-1},\dots,a_{2m-1}+2,\underbrace{2,\dots,2}_{a_{2m}-1})R.
\end{equation}

This equality shows that the first columns of matrices $M^+(a_1,\dots,a_{2m})$ in the left-hand side and $M(a_1+1,\dots)$ in the right-hand side are equal. Therefore, they correspond to expansions of the same number. This proves Theorem~\ref{thm:hirz}.

\subsection{Continued fractions and triangulations}\label{ss:conttriang}

At first glance the connection between coefficients of positive and negative continued fraction expansions described in Theorem~\ref{thm:hirz} appears complex. However, it has a natural interpretation in terms of polygon triangulations. In this section, we present the construction itself, and in the next chapter, we describe its relation to friezes and the  matrices from $\SL_2(\ZZ)$  encountered earlier.

Start with an arbitrary sequence of positive integers $(a_1,\dots,a_{2m})$. Consider two horizontal parallel lines. Draw a triangulation of a polygon whose vertices lie on these lines, with two types of triangles: those pointing upward (with two vertices on the lower line and one on the upper) and pointing downward (with two vertices above and one below). Begin with a segment connecting the lines; let its upper endpoint be the first vertex of our polygon and the lower endpoint the last one. Attach $a_1$ upward-pointing triangles sharing a common vertex to the right of this segment. Then attach $a_2$ downward-pointing triangles to their right, with their shared vertex coinciding with the right vertex of the last triangle from the previous step. Continue adding $a_3$ upward, $a_4$ downward, and so on. Since the total number of such groups is $2m$, that is, even, the last triangle will point downward (see Fig.~\ref{fig:bamboo}). This marks $a_2+a_4+\dots+a_{2m}+1$ points on the upper line and $a_1+a_3+\dots+a_{2m-1}+1$ points on the lower line.

\begin{figure}[h!]
\includegraphics{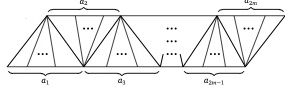}
\caption{Triangulation corresponding to the sequence $(a_1,\dots,a_{2m})$.}\label{fig:bamboo}
\end{figure}

Let $n=a_1+a_2+\dots+a_{2m}+2$. We obtain a triangulation of an $n$-gon of a special kind: it has exactly two ``ears'',  meaning its dual tree (see the proof of Proposition~\ref{prop:algtriang}) has no vertices of degree three. Such a tree (without branching) is called a \emph{bamboo}.

Let $k$ be the number of downward-pointing triangles, i.e., $k=a_2+a_4+\dots+a_{2m}$. The ``ears'' of our triangulation will be at vertices~$k+1$ and~$n$.

The quiddity $(c_1,\dots,c_n)$ of this triangulation can be easily written. Here, $c_1=a_1+1$: the first vertex meets $a_1$ upward and one downward triangle. Next come $a_2-1$ twos: a vertex between two downward triangles meets two triangles. The next vertex has quiddity $a_3+2$ (meeting two downward and $a_3$ upward triangles), and so on until vertex $k+1$ with quiddity~1. The pattern then mirrors: $c_{k+2}=a_{2m}+1$, followed by~$a_{2m-1}-1$ twos, until~$c_n=1$.

\begin{example}\label{ex:triang75} The sequence $[1,2,1,1]$ corresponds to this triangulation of a heptagon:
\[
\xymatrix{
& \bullet\ar@{-}[ddl]\ar@{-}[ddr]\ar@{-}[r] &\bullet \ar@{-}[dd]\ar@{-}[r] &\bullet \ar@{-}[rr]\ar@{-}[ddl]\ar@{-}[ddr] &&\bullet\ar@{-}[ddl]\\
\\
\bullet \ar@{-}[rr]&& \bullet \ar@{-}[rr]&&\bullet\\
}
\]
Its quiddity is $(2,2,3,1,2,4,1)$. 
\end{example}

Thus, the expression in Theorem~\ref{thm:hirz} gives the quiddities of such a triangulation's vertices, but only a ``half'' of them: it includes quiddities of vertices between the two ``ears''. The following exercise shows that this sequence uniquely determines the triangulation.

\begin{exercise}
Given a triangulation of an $n$-gon with two ``ears'' and quiddity $(c_1,\dots,c_n)$, where $c_{k+1}=c_n=1$, show that it is possible to reconstruct $c_1,\dots,c_k$ from $c_{k+2},\dots,c_{n-1}$, and vice versa.
\end{exercise}

Moreover, such a triangulation can be used to reconstruct both the continued fraction and all its convergents, both positive and negative. For this, we define the \emph{Farey\footnote{John Farey Sr. (1766--1826) was an English geologist and polymath. He authored articles on numerous topics ranging from gardening and meteorology to mathematics and music. The notion of Farey sum was introduced during his research on the mathematical theory of music.} sum} operation on fractions.

\begin{definition} For two non-negative irreducible fractions $p/q$ and $r/s$, their \emph{Farey sum} or \emph{mediant} is
\[
\frac{p}{q}\oplus \frac{r}{s}=\frac{p+r}{q+s}.
\]
\end{definition}

The Farey sum is thus the result of ``naive addition of fractions'': that is, adding numerator to numerator and denominator to denominator.

\begin{exercise}
Verify that the Farey sum of two fractions lies between them: if $p/q<r/s$, then $p/q<(p+r)/(q+s)<r/s$.
\end{exercise}

Additionally, we extend the set of non-negative rational numbers with $\infty=1/0$. Clearly, this element can be added in the sense of Farey to others: $p/q\oplus \infty=(p+1)/q$.

Now we describe a construction that computes the continued fraction from the triangulation described above (i.e., from the sequence $(a_1,\dots,a_{2m})$). Place $0=0/1$ at the bottom-left vertex and $\infty=1/0$ at the top-left. Then fill the remaining vertices left to right as follows: if there are fractions in two vertices of a triangle, place their Farey sum in the third vertex. In Example~\ref{ex:triang75}, this gives the following picture:

\[
\label{pic:triang75}
\xymatrix{
& \dfrac 10 \ar@{-}[ddl]\ar@{-}[ddr]\ar@{-}[r] &\dfrac 21 \ar@{-}[dd]\ar@{-}[r] &\dfrac 32 \ar@{-}[rr]\ar@{-}[ddl]\ar@{-}[ddr] &&\dfrac 75 \ar@{-}[ddl]\\
\\
\dfrac 01 \ar@{-}[rr]&& \dfrac 11 \ar@{-}[rr]&&\dfrac 43\\
}
\]

\begin{theorem}\label{thm:triang} Consider the triangulation of an $n$-gon constructed from sequence $(a_1,\dots,a_{2m})$. Let $(c_1,\dots,c_n)$ be its quiddity, where $c_k=c_n=1$. Fill its vertices  with fractions $p_i/q_i$ as stated above. Then we have
\begin{enumerate}
\item $p_{k+1}/q_{k+1}=[a_1,\dots,a_{2m}]=[[c_1,\dots,c_k]]$;

\item $p_i/q_i=[[c_1,\dots,c_{i-1}]]$; i.e., the upper side contains negative convergents to $p_{k+1}/q_{k+1}$;

\item Consider the ``zigzag'' path connecting the bottom-left to top-right vertex via segments between upper and lower sides (shown in Fig.~\ref{fig:bamboo} by bold segments). The sequence of positive convergents lies along this zigzag.
\end{enumerate}
\end{theorem}

We will prove this theorem in the next section using an embedding of the triangulation into the hyperbolic plane and studying the action of continued fraction matrices. Note that part (1) essentially constitutes Theorem~\ref{thm:hirz}, providing an alternative proof of it.

\begin{example} In Example~\ref{ex:triang75}, the positive convergents to $7/5$ are $1/1$, $3/2$, $4/3$, and $7/5$. The negative convergents are $2/1$, $3/2$, and $7/5$.
\end{example}


\section{Triangulations and the Farey graph}

In this section, we will establish a connection between friezes and the group $\PSL_2(\ZZ)$. Our exposition closely follows the paper~\cite{Farey1}, which we recommend to interested readers. We are also grateful to its authors for permission to use the figures below.

\subsection{The Farey graph}

Consider the set of rational numbers $\QQ$, each represented as an irreducible fraction $\dfrac ab$, and add the element $\dfrac 10$ (``infinity''). We obtain the set $\overline\QQ=\QQ\cup\{\infty\}$.

Define the \emph{Farey graph} as an infinite graph whose vertices are elements of $\overline\QQ$. Suppose that two vertices $\dfrac ab$ and $\dfrac cd$ are connected by an edge if and only if $ad-bc=\pm 1$. This is equivalent to $\det\begin{pmatrix}a &c\\b &d\end{pmatrix}=\pm 1$.

If such vertices are connected by an edge, then the matrices $\begin{pmatrix}a &a+c\\b &b+d\end{pmatrix}$ and $\begin{pmatrix}c &a+c\\d &b+d\end{pmatrix}$ also have determinant $\pm 1$, which means that both vertices $\dfrac ab$ and $\dfrac cd$ are connected to the fraction $\dfrac{a+c}{b+d}$, i.e., to their Farey sum.

\begin{exercise} Verify that this fully describes all triangles: every triangle (a triple of pairwise connected vertices) in the Farey graph has the form $\left(\dfrac ab,\dfrac cd, \dfrac{a+c}{b+d}\right)$.
\end{exercise}

\begin{exercise} Verify that if $\dfrac ab$ and $\dfrac cd$ are irreducible fractions with $ad-bc=1$, then their Farey sum $\dfrac{a+c}{b+d}$ is also irreducible.
\end{exercise}

It is convenient to draw the Farey graph on the hyperbolic plane. We will use the upper half-plane model for this; let us briefly recall the main notions related to it.

Let $\HH=\{z\in\CC\mid \Im z>0\}$ be the set of complex numbers with positive imaginary part. We call it the \emph{hyperbolic plane}. The real line, extended with the point $\infty$, is called the \emph{absolute}: it can be thought of as the set of ``points at infinity'' (or ``asymptotic directions") of the hyperbolic plane. The \emph{lines} in $\HH$ are semicircles with centers on the absolute (i.e., perpendicular to the real line), and vertical half-lines starting on the absolute.

We place the vertices of the Farey graph on the absolute and draw each edge as a line (not a segment!) in $\HH$. As we will show in Corollary~\ref{cor:planar}, the graph drawn this way on the upper half-plane is ``planar'': its edges intersect only at vertices.

We will focus not on the entire Farey graph but on its subgraph whose vertices are non-negative rational numbers and $\infty=\dfrac 10$. Additionally, we will depict the hyperbolic plane in the upper half-plane in an unusual way: place the point $\infty$ in the finite part of the absolute and send some negative number to infinity. For example, we can apply the transformation $z\mapsto\dfrac{z}{z+1}$ to the standard upper half-plane model. This maps $0$ to $0$, $1$ to $\dfrac 12$, and infinity to $1$. Thus, the positive half-axis is mapped to the interval $(0,1)$ on the real axis, and the half-plane $\mathrm{Re}\,z>0$ maps to the interior of a semicircle with this interval as its diameter. All edges of the Farey graph connecting positive points on the absolute will be drawn as semicircles inside this semicircle. The subgraph formed by all positive vertices (and $0$ and $\infty$) and their connecting edges is called the \emph{positive Farey graph}. A fragment of it is shown in Fig.~\ref{fig:posfarey}.

\begin{figure}[h!]
\includegraphics{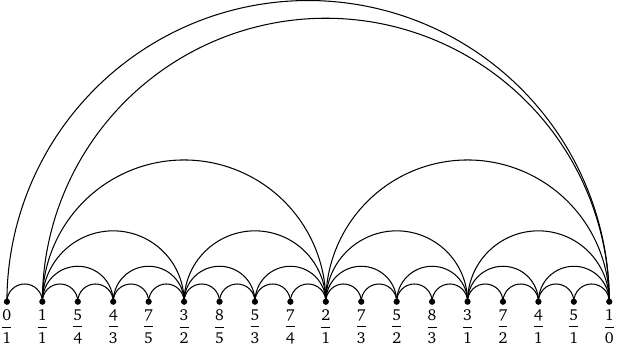}
\caption{A fragment of the positive Farey graph}\label{fig:posfarey}
\end{figure}

For completeness, note that the entire Farey graph can also be nicely depicted in a finite region by applying the transformation $z \mapsto i \dfrac{z-i}{z+i}$. This maps the absolute (the real line) to the unit circle and the upper half-plane to the unit disk. Here, $\infty$ is mapped to $i$, $0$ to $-i$, and the points $-1$ and $1$ remain fixed. The arcs perpendicular to the absolute (generalized circles) map to arcs perpendicular to the absolute, resulting in the picture shown in Fig.~\ref{fig:fareyfull}.

\begin{figure}[h!]
\includegraphics[scale=0.7]{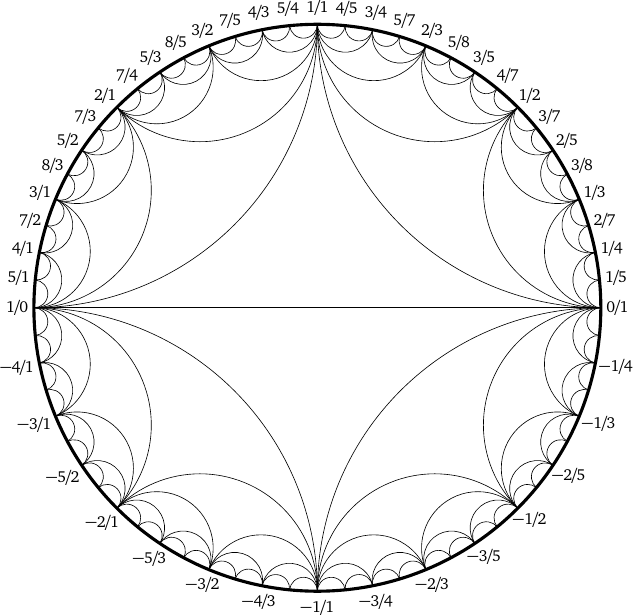}
\caption{The Farey graph in the disk}\label{fig:fareyfull}
\end{figure}

\subsection{Embedding triangulations into the positive Farey graph}

Triangulations of polygons can be embedded into the Farey graph drawn on the plane with the vertices of the polygon lying on the absolute. Consider a triangulation of an $n$-gon and mark one of its sides; assume this side connects the $n$-th and first vertices. Write $0=\dfrac 01$ and $\infty=\dfrac 10$ at its endpoints. Then assign to each vertex a positive rational number using an algorithm similar to that in~\S\,\ref{ssec:friezetriang}: for each triangle where two vertices already have numbers, write their Farey sum in the third vertex.

Now mark all $n$ resulting rational numbers on the absolute of the hyperbolic plane and draw all possible edges of the Farey graph between them; this will reconstruct the original triangulation.

This method of assigning numbers to triangulation vertices already appeared in~\S\,\ref{ss:conttriang} for the special case of ``bamboo'' triangulations: here, the marked side was the left side of the triangulation.

\begin{example} Consider the  triangulation of a heptagon shown on p.~\pageref{pic:triang75}. Its embedding into the Farey graph is shown below.
\end{example}

\begin{figure}[h!]
\includegraphics{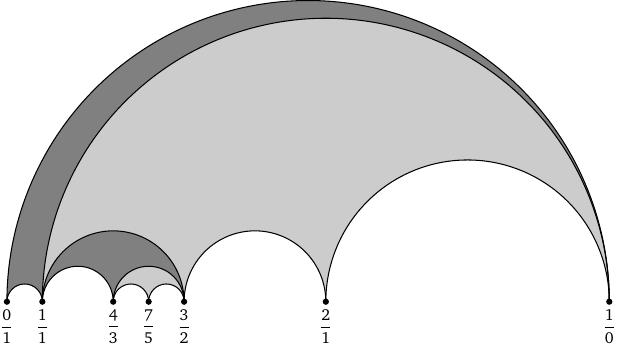}
\caption{Triangulation corresponding to the fraction $\dfrac 75$}\label{fig:mainexample}
\end{figure}

\begin{prop} Suppose the vertices of the $n$-gon are labeled with irreducible fractions $\dfrac{u_1}{v_1}=\dfrac 10, \dfrac{u_2}{v_2},\dots,\dfrac{u_n}{v_n}=\dfrac 01$. Then the frieze corresponding to this triangulation has adjacent diagonals $v_2,\dots,v_n$ and $u_1,u_2,\dots,u_{n-1}$.
\end{prop}

\begin{proof}
By computing Farey sums of fractions, we essentially apply the algorithm from \S\,\ref{ssec:friezetriang} twice: separately to numerators and denominators. Thus, the sequence of numerators will be a diagonal of the frieze corresponding to the initial data $u_1=1,u_n=0$, and the sequence of denominators will correspond to the initial data $v_1=0$, $v_n=1$.
\end{proof}

\begin{example} The triangulation from the previous example corresponds to the diagonals of the following frieze.
\begin{figure*}[h!]
\includegraphics{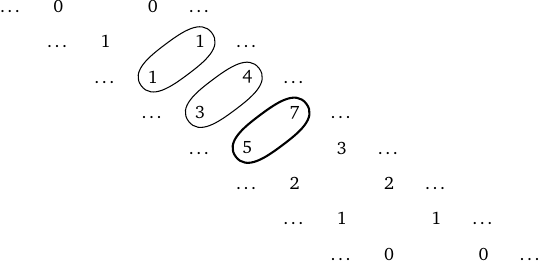}    
\end{figure*}

Thus, the numerator and denominator of each fraction at the vertices appear along the antidiagonal. In particular, the negative convergents to $\dfrac 75$, namely $\dfrac 11$, $\dfrac 43$, and $\dfrac 75$ itself, can be read along the antidiagonals. 
\end{example}

\subsection{The action of the group $\PSL_2(\ZZ)$}

Consider the group of real $2\times 2$ matrices with determinant 1, denoted by $\SL_2(\RR)$. This group acts on $\HH$ as follows: the element $\begin{pmatrix}a&b\\c&d\end{pmatrix}$ maps the point $z\in\HH$ to $\dfrac{az+b}{cz+d}$.

\begin{exercise} Verify that this is indeed an action, i.e., the upper half-plane is mapped to itself, and matrix multiplication corresponds to composition of mappings.
\end{exercise}

Additionally, $\SL_2(\RR)$ contains a non-trivial element acting as the identity on $\HH$: the matrix $-I$. Thus, we can consider the quotient group $\SL_2(\RR)/\{\pm I\}$, denoted $\PSL_2(\RR)$, that also acts on $\HH$. Its elements can be viewed as $2\times 2$ matrices taken up to sign: $\begin{pmatrix}a&b\\c&d\end{pmatrix}$ and $\begin{pmatrix}-a&-b\\-c&-d\end{pmatrix}$ correspond to the same element in the quotient group.

Finally, we consider not the entire group $\PSL_2(\RR)$ but rather its subgroup $\PSL_2(\ZZ)$, consisting of all integer matrices in $\PSL_2(\RR)$. Clearly, it maps rational points on the absolute to rational points. Moreover, it preserves the Farey graph: two points $p/q$ and $r/s$ are connected by an edge if and only if the matrix $\begin{pmatrix}p&r\\q&s\end{pmatrix}$ has determinant $\pm 1$; this means that for any matrix $\begin{pmatrix}a&b\\c&d\end{pmatrix}\in\PSL_2(\ZZ)$, the product $\begin{pmatrix}a&b\\c&d\end{pmatrix}\begin{pmatrix}p&r\\q&s\end{pmatrix}$ has the same determinant. This is equivalent to saying the images of these points, $\dfrac{ap+bq}{cp+dq}$ and $\dfrac{ar+bs}{cr+ds}$, are also connected by an edge. Thus, the group $\PSL_2(\ZZ)$ preserves the Farey graph.

\begin{remark} This action can be viewed as follows: the group $\SL_2(\RR)$ naturally acts on $\CC^2$. By associating to a non-zero vector $(z_1, z_2)\in \CC^2$ its ``slope'' $z=\dfrac{z_1}{z_2}\in \CC \cup\{\infty\}$ (in other words, considering the action of $\SL_2(\RR)$ on lines through the origin), we obtain the described action. In these terms, the absolute corresponds to the real subset $\RR^2\subset \CC^2$ preserved by $\SL_2(\RR)$. Here, a pair of vertices $\dfrac pq$ and $\dfrac rs$ connected by an edge in the Farey graph corresponds to the pair of vectors $(p,q)$ and $(r,s)\in \ZZ^2$ with $\det\left(\begin{matrix}p &r\\q&s\end{matrix}\right)=1$. Of course, a matrix $A\in \SL(2,\ZZ)$ maps such a pair of vectors to another such pair.
\end{remark}

\begin{prop}\label{pr:transit} The group $\PSL_2(\ZZ)$ acts transitively and faithfully on the oriented edges of the Farey graph: any edge (with marked start and end) can be mapped to any other by a unique transformation.
\end{prop}

\begin{proof}
We show that the pair $(\infty,0)$ can be mapped to the pair $\left(\dfrac pq,\dfrac rs\right)$, where $ps-qr=1$, by a linear-fractional transformation with integer coefficients. Indeed, this transformation is given by the matrix $\begin{pmatrix}p &r\\q&s\end{pmatrix}$. The faithfulness of the action is also clear: it suffices to check that the stabilizer of the pair $(\infty,0)$ is trivial. Obviously, it consists only of $\pm I$.
\end{proof}

From this, we easily derive the previously mentioned result about the planarity of the Farey graph.

\begin{corollary}\label{cor:planar} The edges of the Farey graph do not intersect.
\end{corollary}

\begin{proof} Due to the transitivity of the action, it suffices to show that the edge connecting $0$ and $\infty$ does not intersect any other edge. Indeed, such another edge would have to connect a positive and a negative point $\dfrac ab>0>-\dfrac cd$ (where $a,b,c,d$ are positive integers). But then the determinant $ad-b(-c)=ad+cb\geq 2$ cannot equal 1.
\end{proof}

\begin{corollary}\label{cor:transit}
Consider two triangulations of an $n$-gon with identical quiddities and a marked side in each of them, embedded into the Farey graph. Then there exists a unique element of $\PSL_2(\ZZ)$ that maps one triangulation to the other and preserves the marked side.
\end{corollary}

\begin{proof} By Proposition~\ref{pr:transit}, there is a unique element of $\PSL_2(\ZZ)$ that maps the marked sides into each other, and such that the half-plane containing the first polygon maps to the half-plane containing the second polygon. The remaining vertices of the triangulation are uniquely determined by traversing the triangulation tree: the third vertex of each triangle is obtained as the Farey sum of the first two.
\end{proof}

\subsection{Generators of the group $\PSL_2(\ZZ)$}

In Section~\ref{ssec:sl2z}, we saw that any two of the following three matrices
\[
R=\begin{pmatrix}1 & 1\\ 0 &1
\end{pmatrix},\qquad
L=\begin{pmatrix}1 & 0\\ 1 &1
\end{pmatrix},\qquad
S=\begin{pmatrix}0 & -1\\ 1 &0
\end{pmatrix}
\]
generate $\SL_2(\ZZ)$. Thus, their images generate $\PSL_2(\ZZ)$. With some abuse of notation, we will use the same symbols for elements of $\PSL_2(\ZZ)$ as in $\SL_2(\ZZ)$, i.e., before factorization.

Let us see how these transformations act on $\HH$. The operator $R$ maps $z$ to $z+1$. It fixes $\infty$ and shifts the absolute by one to the right. The transformation $R^a=\begin{pmatrix}1&a\\0&1\end{pmatrix}$ shifts the absolute by $a$ units, where $a$ can be any integer (possibly negative).

The operator $S=-\dfrac 1z$ is the composition of inversion and reflection across the imaginary axis. This transformation is involutive: $S^2=I$. Note that in $\SL_2(\ZZ)$, the matrix $\begin{pmatrix}0&-1\\1&0\end{pmatrix}$ has order 4, not 2.

The operator $L$ fixes $0$ and maps $z=\dfrac 1n$ to $\dfrac 1{n+1}$. It is easy to check (do this!) that $LS=SR^{-1}$.

\subsection{Stern--Brocot sequences}

Now consider the line $C$ connecting $\infty$ and $0$ (as before, we consider oriented lines; in the figure below, the arrow points from the start to the end of the line). Applying $R$ to it, we get the line $RC$ starting at $\infty$ and ending at $1$; similarly, $LC$ connects $1$ and $0$ on the absolute. We obtain a triangle bounded by the lines $C$, $LC$, and $RC$. Call this triangle $T$. The coordinate of its middle vertex is the Farey sum of the two extreme vertices.

Next, apply $L$ to $T$. We obtain the triangle $LT$ with vertices at $0$, $\dfrac 12$, and $1$. The sides of this triangle are $LC$, $L^2C$, and $LRC$. Similarly, the triangle $RT$ has sides $RC$, $RLC$, and $R^2C$ (see Fig.~\ref{fig:sterntriang}). This gives the following set of vertices on the absolute: $0$, $\dfrac 12$, $1$, $2$, $\infty$.

\begin{figure}[h!]
\includegraphics[width=10cm]{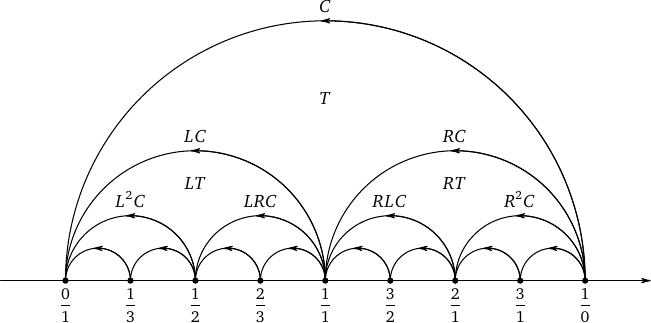}
\caption{Stern--Brocot triangulation}\label{fig:sterntriang}
\end{figure}

This procedure can be continued: apply $L$ and $R$ to $LT$ and $RT$. The result is a set of nine vertices: $0=\dfrac 01$, $\dfrac 13$, $\dfrac 12$, $\dfrac 23$, $\dfrac 11$, $\dfrac 32$, $\dfrac 21$, $\dfrac 31$, $\dfrac 10=\infty$, where each pair of consecutive vertices is connected by one of the lines obtained from $C$ by applying $L$ and $R$ in all possible orders. Note that each new vertex on the absolute is the Farey sum of its neighbors.

This motivates the following definition.

\begin{definition} The \emph{zeroth Stern\footnote{Moritz Abraham Stern (1807--1894) was a German mathematician specializing in number theory. A student of Carl Friedrich Gauss, he introduced the concept of a Stern prime: a prime number that cannot be expressed as the sum of a smaller prime and twice the square of an integer. He defined the Stern--Brocot tree in an 1858 publication, independently of Achille Brocot.}--Brocot\footnote{Achille Brocot (1817--1878) was a French clockmaker and amateur mathematician. A co-founder of the Parisian clockmaking firm ``Brocot \& Delettrez'', he authored several improvements to clock mechanisms. Independently of Stern, he devised the Stern--Brocot tree for approximating real numbers with rational numbers, which played an important role in calculating the parameters of gears used in clock mechanisms.} sequence} is the sequence of fractions $\left(\dfrac 01, \dfrac 11, \dfrac 10\right)$. The \emph{$k$-th Stern--Brocot sequence} is obtained from the $(k-1)$-th sequence by inserting the mediant between each pair of consecutive terms.
\end{definition}

Thus, the first Stern--Brocot sequence is equal to $\left(\dfrac 01,\dfrac 12,\dfrac 11,\dfrac 21, \dfrac 10\right)$, the second sequence is $\left(\dfrac 01,\dfrac  13, \dfrac 12, \dfrac 23, \dfrac 11,\dfrac 32,\dfrac 21, \dfrac 31,\dfrac 10\right)$, and so on.

\begin{remark}
Here, as before, we treat infinity as an admissible ``number,'' expressed as the irreducible ``fraction'' $\dfrac 10$. This will not cause any issues since we will only compare fractions (considering infinity greater than any number) and take their Farey sums. We will not further elaborate on this point.
\end{remark}

Clearly, each Stern--Brocot sequence is increasing: this follows from the fact that the mediant of two fractions lies between them. In particular, all terms in each sequence are distinct.

\begin{theorem}\label{thm:stern}
\begin{enumerate} 
\item Any two fractions $\dfrac pq<\dfrac rs$ that are neighbors in some Stern--Brocot sequence satisfy the relation $qr-ps=1$, i.e., the corresponding numbers are connected by an edge in the Farey graph;

\item all fractions appearing in Stern--Brocot sequences are irreducible;

\item every irreducible fraction $\dfrac ab>0$ appears in some Stern--Brocot sequence.
\end{enumerate}
\end{theorem}

\begin{proof}
(1) We prove this by induction. For the zeroth sequence, the claim is obvious. When we insert the mediant $\dfrac {p+r}{q+s}$ between two neighbors $\dfrac pq$ and $\dfrac rs$, the required relations hold for it as well:
\[
(p+r)q-(q+s)p=qr-ps=1;\qquad r(q+s)-s(p+r)=qr-ps=1.
\]

(2) This follows from (1): if $qr-ps=1$, then $p$ and $q$ are coprime.

(3) We show that every irreducible fraction $\dfrac ab$ eventually appears in some Stern--Brocot sequence. Suppose it does not appear in a certain sequence. Find the closest neighbors $\dfrac pq$ and $\dfrac rs$ in this sequence between which $\dfrac ab$ lies:
\[
\dfrac pq< \dfrac ab< \dfrac rs.
\]
Proceed to the next sequence by replacing one endpoint of $\left[\dfrac pq;\dfrac rs\right]$ with the mediant of its endpoints and choosing the half containing $\dfrac ab$. This process cannot continue indefinitely because the conditions
\[
\dfrac ab-\dfrac pq>0 \qquad\text{and}\qquad \dfrac rs-\dfrac ab>0
\]
imply that
\[
aq-bp\geq 1 \qquad\text{and}\qquad br-as\geq 1.
\]
Thus, we have
\begin{equation}\label{eq:stern}
(r+s)(aq-bp)+(p+q)(br-as)\geq p+q+r+s.
\end{equation}

On the other hand, the left-hand side of (\ref{eq:stern}) equals
\[
(r+s)(aq-bp)+(p+q)(br-as) = a((r+s)q-(p+q)s) + b((p+q)r-(r+s)p) = a+b.
\]
Thus, (\ref{eq:stern}) is equivalent to
\[
a+b\geq p+q+r+s.
\]
But at each step, either $p$, $q$, $r$, or $s$ increases, so this inequality will fail after at most $a+b$ steps.
\end{proof}

Returning to Fig.~\ref{fig:sterntriang}, Theorem~\ref{thm:stern} implies that every positive rational number appears as a vertex of some triangle in the Stern--Brocot triangulation.

This picture can also be interpreted as follows. Consider the infinite graph dual to the Stern--Brocot triangulation: assign to each triangle a vertex labeled by the ``middle'' vertex of the triangle. Connect vertices corresponding to adjacent triangles. We obtain a binary tree, the beginning of which is shown in Fig.~\ref{fig:sterntree}; it is also called the \emph{Stern--Brocot tree}.

\begin{figure}[h!]
\includegraphics{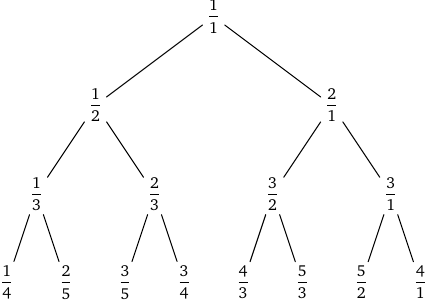}
\caption{Stern--Brocot tree}\label{fig:sterntree}
\end{figure}

Theorem~\ref{thm:stern} states that every positive rational number appears in the Stern--Brocot tree exactly once. Thus, each number can be encoded by a path from the root (i.e., from $\dfrac 11$) to the corresponding vertex: each number corresponds to a unique word in the letters $R$ and $L$. For example, $\dfrac 21$ corresponds to $R$, $\dfrac 43$ to $RLL$, and so on. This yields the following corollary.

\begin{corollary}\label{cor:stern} For every positive irreducible fraction $\dfrac{p}{q}>1$, there exists a sequence of exponents $(a_1,\dots,a_{2m})$ such that
\[
\begin{pmatrix} p\\ q\end{pmatrix}=R^{a_1}L^{a_2}\dots R^{a_{2m-1}}L^{a_{2m}}\begin{pmatrix} 1\\ 0\end{pmatrix}=M^+(a_1,\dots,a_{2m})\begin{pmatrix} 1\\ 0\end{pmatrix},
\]
with all $a_i>0$.
\end{corollary}

\begin{proof}
Note that $\begin{pmatrix} 1\\ 1\end{pmatrix}=L\begin{pmatrix} 1\\ 0\end{pmatrix}$ corresponds to the point $1$ on the absolute. Theorem~\ref{thm:stern} states that any positive rational number on the absolute can be obtained from $1$ by applying a word in $R$ and $L$. Since $\dfrac pq>1$, this number lies in the right half of the Stern--Brocot graph, so the first letter of the corresponding word is $R$, not $L$.
\end{proof}

To summarize: we have figured out the geometric meaning of the operators $R$ and $L$. They map the entire Stern--Brocot tree to its right and left halves, respectively. Thus, the sequence $(a_1,\dots,a_{2m})$ defines the transformation $M^+(a_1,\dots,a_{2m})$, which maps the line $(\infty, 0)$ to $\left(\dfrac{p}{q}, \dfrac{r}{s}\right)$, where $\dfrac rs$ is the last convergent to $\dfrac{p}{q}$. The resulting triangles define a ``bamboo'' triangulation of a polygon with the marked side $(\infty,0)$, consisting of $a_1$ upward-pointing triangles, $a_2$ downward-pointing triangles, $a_3$ upward-pointing triangles, and so on.

\subsection{Rotating a triangulation}

Consider a triangulation of an $n$-gon embedded in the Farey graph. Let $c_1$ be the quiddity of the first vertex, i.e., the number of triangles meeting at the vertex labeled $0$. This means the second vertex is labeled by $\dfrac 1{c_1}$ (why?). Consider the operator $\begin{pmatrix} c_1 & -1\\1 &0\end{pmatrix}$ and apply it to $\HH$. Let us see where our $n$-gon is mapped. First, consider the images of the first two vertices, $0$ and $\dfrac 1{c_1}$:
\[
\begin{pmatrix} c_1 & -1\\1 &0\end{pmatrix}\begin{pmatrix}0\\ 1\end{pmatrix}= \begin{pmatrix}-1\\ 0\end{pmatrix}=\infty;\qquad 
\begin{pmatrix} c_1 & -1\\1 &0\end{pmatrix}\begin{pmatrix}1\\ c_1\end{pmatrix}= \begin{pmatrix}0\\ 1\end{pmatrix}=0.
\]
Thus, they are mapped to the last and first vertices, respectively, and the line connecting them is mapped to $(\infty,0)$.

The remaining vertices are also mapped to some vertices with positive coordinates (check: why positive?). Generally, they will not coincide with the vertices of the original $n$-gon. However, their relative positions on the line are preserved since the absolute is preserved by this transformation.

Let $\alpha_1=0,\alpha_2=\dfrac 1{c_1}, \dots,\alpha_n=\infty$ be the vertices of the original polygon, where $p_i\in\QQ_{\geq 0}\cup\{\infty\}$. Let $\beta_i$ be the image of $\alpha_i$ under the action of $\begin{pmatrix} c_1 & -1\\1 &0\end{pmatrix}$. Our $n$-gon is mapped to an $n$-gon with vertices $\beta_1=\infty, \beta_2=0,\dots,\beta_{n}=c$.

This operation corresponds to rotating the $n$-gon. We have proved the following proposition.

\begin{prop}\label{prop:turn} The transformation $\begin{pmatrix} c_1 & -1\\1 &0\end{pmatrix}$ maps an $n$-gon with side $(\infty,0)$ and quiddity $(c_1,\dots,c_n)$ (where the first vertex is $0$) to an $n$-gon with side $(\infty,0)$ and quiddity $(c_n,c_1,\dots,c_{n-1})$.
\end{prop}

This implies the following corollary.
\begin{corollary}
Let $(c_1,\dots,c_n)$ be the quiddity of an $n$-gon. Then
\[
\begin{pmatrix} c_1 & -1\\1 &0\end{pmatrix}\dots \begin{pmatrix} c_n & -1\\1 &0\end{pmatrix}=\pm I.
\]
\end{corollary}

\begin{proof}
Consider an $n$-gon with marked side $(\infty,0)$ and quiddity $(c_1,\dots,c_n)$. By the previous proposition, the composition of such transformations maps it to an $n$-gon with the same marked side and quiddity. According to Corollary~\ref{cor:transit}, this composition is the identity element in $\PSL_2(\ZZ)$, i.e., $\pm I$.
\end{proof}

Another corollary of Proposition~\ref{prop:turn} can be obtained by considering a ``bamboo'' triangulation with quiddity $(c_1,\dots,c_k,1,c_{k+2},\dots,1)$. Let $\dfrac pq$ be the coordinate of the $(k+1)$-th vertex. Consider the operator that rotates our triangulation $k$ times. This operator is equal to
\[
\begin{pmatrix}c_1 &-1\\1& 0\end{pmatrix}\dots \begin{pmatrix}c_k &-1\\1& 0\end{pmatrix}=M(c_1,\dots,c_k).
\] 
It maps the first vertex of the triangulation, $\infty$, to the $(k+1)$-th vertex, $\dfrac pq=[[c_1,\dots,c_k]]$. The last vertex is mapped to the $k$-th vertex, $\dfrac {\tilde{r}}{\tilde{s}}=[[c_1,\dots,c_{k-1}]]$.

Thus, we have shown that if the sequences $(a_1,\dots,a_{2m})$ and $(c_1,\dots,c_k)$ are constructed from the same ``bamboo'' triangulation, then the operators $M^+(a_1,\dots,a_{2m})$ and $M(c_1,\dots,c_k)$ map $\infty$ to the same point. Hence, the first columns of the corresponding matrices coincide (up to sign), which again proves Theorem~\ref{thm:hirz}.

\begin{exercise}
Prove Proposition~\ref{pr:hirzlite} similarly.
\end{exercise}


\section{Summary and further steps}

\subsection{Continued fractions and friezes as solutions to equations in $\SL_2(\ZZ)$}

In the previous chapters, we discussed friezes, the expansions of rational numbers into continued fractions, and the related identities in the groups $\PSL_2(\ZZ)$ and $\SL_2(\ZZ)$. 

We saw that for a positive rational number $\dfrac pq$, we can construct its positive and negative continued fraction expansions; these expansions correspond to a ``bamboo'' triangulation of some $n$-gon. If the quiddity of this $n$-gon is $(c_1,\dots,c_k,1,c_{k+2},\dots,c_{n-1},1)$, then $\dfrac pq=[[c_1,\dots,c_k]]$. Moreover, the quiddity of  a triangulation of the $n$-gon provides a solution to the matrix equation
\[
M(c_1,\dots,c_n)=-I.
\]

More generally, such a solution holds for the quiddity of an arbitrary triangulation of an $n$-gon, not necessarily a ``bamboo''. Such triangulations define friezes of order $n$. The elements of friezes are constructed based on the quiddity of the triangulation as continuants. In particular, if the triangulation was a ``bamboo'' and corresponded to the number $\dfrac pq$, then the numerators and denominators of the convergents to $\dfrac pq$ can be read on the adjacent diagonals of the frieze.

\subsection{Triangulations and $3d$-dissections} 

A question arises: for which sets $(c_1,\dots,c_n)$ is the matrix $M(c_1,\dots,c_n)$ equal to $\pm I$, i.e., defines the identity transformation in $\PSL_2(\ZZ)$? Do all such sets arise from triangulations? Obviously not: for example, one can insert three consecutive ones anywhere in an existing solution and obtain a new solution by using the relation $M(1,1,1)=-I$. So how can we describe all solutions?





A complete answer to this question was obtained in~2018. It is due to Valentin\,Ovsienko~\cite{Ovsienko18}. To formulate it, let us first define the concept of a \emph{$3d$-dissection} of a convex polygon.

\begin{definition}
A \emph{$3d$-dissection} of a convex $n$-gon (with numbered vertices) is its partition by a set of non-intersecting diagonals into polygons, such that has the number of sides of each of these polygons is divisible by~3. The \emph{quiddity} of a $3d$-dissection of an~$n$-gon is the set of numbers $(c_1,\dots,c_n)$, where $c_i$ is the number of polygons incident to the $i$-th vertex.
\end{definition}

Thus, instead of triangulations, we will consider dissections of an $n$-gon into triangles, hexagons, nonagons, etc. Unlike the case of triangulations, a $3d$-dissection of an $n$-gon cannot be reconstructed from its quiddity.

\begin{exercise}
Give an example of two distinct $3d$-dissections of an $n$-gon with the same quiddity.
\end{exercise}

\begin{hint}
It suffices to take $n=8$.
\end{hint}

It turns out that, in addition to triangulations, solutions to the equation $M(c_1,\dots,c_n)=\pm I$ can also arise from other $3d$-dissections of the $n$-gon. Namely, the following theorem holds.

\begin{theorem}[\cite{Ovsienko18}] The equality $M(c_1,\dots, c_n)=\pm I$ holds if and only if $(c_1,\dots,c_n)$ is the quiddity of some $3d$-dissection of an $n$-gon. Moreover, $M(c_1,\dots, c_n)=- I$ if the number of polygons in the dissection with an even number of sides is even, and $M(c_1,\dots, c_n)=I$ otherwise.
\end{theorem}



\end{document}